\documentclass[12pt]{amsart}

\usepackage{amssymb,amsfonts,amsthm}
\usepackage{amsmath}
\usepackage[all]{xy}

\usepackage[margin=3cm]{geometry}

\numberwithin{equation}{section}
\newtheorem{thm}[equation]{Theorem}
\newtheorem{lem}[equation]{Lemma}
\newtheorem{cor}[equation]{Corollary}
\newtheorem{prop}[equation]{Proposition}

\theoremstyle{definition}
\newtheorem{dfn}[equation]{Definition}

\theoremstyle{remark}
\newtheorem{rem}[equation]{Remark}

\renewcommand{\leq}{\leqslant}
\renewcommand{\geq}{\geqslant}
\renewcommand{\div}{\mathop{\mathrm{div}}\nolimits}

\newcommand{\abs}[1]{\left\lvert#1\right\rvert}

\newcommand{\norm}[1]{\left\|#1\right\|}
\newcommand{\MA}{\mathop{\mathrm{MA}}\nolimits}
\newcommand{\pigwedge}{\mathop{\bigwedge}\nolimits}
\newcommand{\PSH}{\mathop{\mathrm{PSH}}\nolimits}
\newcommand{\soup}{\mathop{\mathrm{sup^*}}}
\newcommand{\Supp}{\mathop{\mathrm{Supp}}\nolimits}
\newcommand{\Vol}{\mathop{\mathrm{Vol}}\nolimits}

\begin{document}

%タイトル%
\title[Restricted Bergman kernel asymptotics]{Restricted Bergman kernel asymptotics}
\author[Tomoyuki Hisamoto]{Tomoyuki Hisamoto}
\address{Graduate School of Mathematical Sciences, The University of Tokyo, 
3-8-1 Komaba Meguro-ku, Tokyo 153-0041, Japan}
\email{hisamoto@ms.u-tokyo.ac.jp}
\date{}
\maketitle

%アブストラクト%
\begin{abstract}

In this paper, we investigate a {\em restricted } version of Bergman kernels 
for high powers of a big line bundle over a smooth projective variety. 
The geometric meaning of the leading term is specified. 
As a byproduct, we derive some integral representations for the restricted volume.

\end{abstract}

%１イントロダクション%　　　　　　　　　　
\section{Introduction} 
The subjects discussed in this paper originate from the extension problem. 
Let $L$ be a big line bundle on a smooth complex projective variety $X$ and 
$Z \subseteq X$ a subvariety. 
Denote by $\iota:Z \hookrightarrow X$ the inclusion map. 
We always use this notation unless specifically noted. 
It is important to know 
how many sections of $L|_Z$ are extended to the ambient space $X$. 
We can expect to get more such sections taking high tensor powers of $L$,   
thus we are led to consider the spaces of sections 
\begin{equation*} 
 H^0(X|Z, \mathcal{O}(mL)):= 
        \mathrm{Im} \big[ \iota^* : H^0(X, \mathcal{O}(mL)) \to H^0(Z, \mathcal{O}(mL)) \big].
\end{equation*} 
The restricted volume 
\begin{equation*} 
 \Vol_{X|Z}(L)
  :=\limsup_{m \to \infty} \frac{\dim H^0(X|Z, \mathcal{O}(mL))}{m^p/p!}
\end{equation*} 
measures the asymptotic growth of these spaces. 
Here $p$ denotes the complex dimension of $Z$. 
The notion of the restricted volume first appeared in Tsuji's paper \cite{Tsu06} 
(see also \cite{HM06}, \cite{Tak06}, \cite{ELMNP09}). 
In this paper, we investigate a local version of the restricted volume. 
\begin{dfn}  
  Let $h_L$ be a smooth Hermitian metric on $L$, $\varphi \in C^{\infty}(X; \mathbb{R})$ a smooth weight, 
  and $d \mu $ a volume form on $Z$. 
  Then for any positive integer $m$, {\em the restricted Bergman kernel} 
  of $(Z, mL, h_L^me^{-m\varphi}, d\mu)$ is defined as follows:  
\begin{equation*} 
  B_{X|Z}(m\varphi) := \abs{s_{m,1}}_{ m \varphi}^2+...+\abs{s_{m,N(m)}}_{m \varphi}^2. 
\end{equation*} 
 Here $\{s_{m,1},...,s_{m,N(m)}\}$ is a complete orthonormal system of
 $H^0(X|Z, \mathcal{O}(mL))$ 
 with respect to the norm 
 \begin{equation*}
  \norm{s}_{m \varphi}^2 := \int_Z \abs{s}_{m \varphi}^2 d \mu, 
 \end{equation*}
 \begin{equation*}
  \abs{s}_{m \varphi}^2 := \iota^*h_L^m(s,s)e^{-m \iota^*\varphi}.  
 \end{equation*}
  $\hfill \Box$
\end{dfn}
By definition, $B_{X|Z}(m\varphi)$ is a smooth function on $Z$ and 
\begin{equation*}
 \int_Z B_{X|Z}(m \varphi ) d \mu = \dim H^0(X|Z, \mathcal{O}(mL)). 
\end{equation*}
In fact $B_{X|Z}(m \varphi)$ tells not only the dimension 
but rather deeper information of the space of sections. 
The study of the asymptotic behavior of $B_{X|Z}(m \varphi)$ is itself an important problem in complex geometry. 

In this paper we closely examine the leading term of $B_{X|Z}(m \varphi)$. 
If $L$ is ample and the metric $h_Le^{-\varphi}$ has the positive curvature,   
then by the Serre vanishing theorem the problem is reduced to the case $Z=X$. 
In this case, Tian's classical result (\cite{Tian90}) gives the complete answer.   
For general $L$ and $h_L$, Berman first treated the case $Z=X$ 
in \cite{Ber09}. And in that paper, he also mentioned the restricted case without proof. 
Without the assumption of curvature positivity, the effect of the subvariety can not be ignored. 
We give a complete picture in the restricted case and specify the limit of $m^{-p}B_{X|Z}(m \varphi)$. 
Our study can be seen as a local version of 
the restricted Fujita-type approximation (Theorem \ref{Fujita approximation}). 
As a result, a localization of $\Vol_{X|Z}(L)$ is given.  

To state our results, we need some notion which arises in non-positive curvature case. 
First denote by $\mathbb{B}_+(L) \subsetneq X$ the augmented base locus (see \cite{Laz04}, Definition 10.3.2). 
This is actually an algebraic subset of $X$ and $L$ is ample precisely if $\mathbb{B}_+(L) = \emptyset$.   
Secondly, we denote by $P_{X|Z}\varphi$ the equilibrium weight associated to $\varphi$ (see Definition \ref{definition of equilibrium weight}). 
Let $\theta = -dd^c \log h_L$ be the Chern curvature of $h_L$, then $\theta + dd^c P_{X|Z}\varphi$ defines a positive current on $Z$, 
and $P_{X|Z}\varphi=\varphi$ holds if $\theta + dd^c \varphi$ is positive. 
Roughly speaking, $P_{X|Z}\varphi$ is the best $\theta$-plurisubharmonic function on $Z$ approximating $\varphi$. 
Further, one can measure the rest of the positivity of $\theta + dd^c \varphi$ by 
$\MA(P_{X|Z}\varphi) := \langle (\theta + dd^c P_{X|Z}\varphi)^p \rangle$, 
the non-pluripolar Monge-Amp\`{e}re product of $P_{X|Z}\varphi$
(see Definition \ref{definition of MA}). 
\begin{thm}\label{Main Theorem} 
   Assume $Z$ is smooth and $Z \nsubseteq \mathbb{B}_+(L) $. 
   Then the convergence  
  \begin{equation*}
   \frac{B_{X|Z}(m \varphi)}{m^p/p!}d \mu \to \MA(P_{X|Z} \varphi) \\ 
  \end{equation*}
  holds in the sense of currents. 
\end{thm} 
As byproducts of our investigation of restricted Bergman kernel asymptotics, 
we can get several integral representations of restricted volumes (discussed in section 4).  
For instance, we have the following. 
\begin{thm}\label{Byproducts}
  In the situation of Theorem \ref{Main Theorem}, the following holds. 
 \begin{align*}
  \Vol_{X|Z}(L) 
  &=\int_{Z} \MA(P_{X|Z} \varphi) 
  =\int_{Z} \MA(\iota^*P_X \varphi) \\
  &=\sup_T \int_{Z} \big\langle (\iota^*T)^p \big\rangle 
  = \int_{Z} \big\langle (\iota^*T_{\min})^p \big\rangle, 
 \end{align*}
  where $T$ runs through all the closed positive currents in $c_1(L)$, 
  with small unbounded loci not contained in $\iota(Z)$. 
  We denote by $T_{\min}$ a minimum singular closed positive current in $c_1(L)$. 
\end{thm}
These formulas can be seen as generalizations of the main result of \cite{Bou02}. 
If $L$ is ample, $\Vol_{X|Z}(L)$ equals to the intersection number $(L^p.Z)$ which plays an important role in many geometric questions. 
But for general line bundles, these intersection numbers do not work well to describe function-theoretic properties of $L$. 
Our results indicate that $\Vol_{X|Z}(L)$ is the natural generalization of $(L^p.Z)$ for general line bundles.  

Let us explain the point of our proof of Theorem \ref{Main Theorem}. 
We basically follows Berman's approach 
but there are two difficulties in the restricted case.   
First, to deal with general subvarieties, we need a variant of $L^2$-extension theorems. 
The desired extension theorem is the following. 

\begin{thm}\label{L2 extension 2}
  Let $X$ be a smooth projective variety, $Z \subseteq X $ a smooth subvariety, 
  $\omega$ a fixed K\"{a}hler form, and 
  $E \to X$ a holomorphic vector bundle with a smooth Hermitian metric $h_E$. 
  Then there exist constants $N=N(Z, X, h_E, \omega)$ and $C=C(Z, X)>0$ such that 
  the following holds. 
  
  Let $L \to X$ be a holomorphic line bundle with a singular Hermitian metric $h_Le^{-\varphi}$ 
  such that its Chern curvature satisfies
   \begin{equation*}
     \theta + dd^c \varphi \geq N \omega. 
   \end{equation*} 
   Then for any section $s \in H^0(Z, \mathcal{O}(E \otimes L)) $ with 
   \begin{equation*}
     \int_Z \abs{s}^2e^{-\varphi}dV_{\omega, Z}  < + \infty, 
   \end{equation*}
  there exists a section $ \widetilde{s} \in H^0(X, \mathcal{O}(E \otimes L)) $
  such that $\widetilde{s}|_Z=s$ and
   \begin{equation*}
      \int_X \abs{\widetilde{s}}^2e^{-\varphi}dV_{\omega, X} 
      \leq C \int_Z  \abs{s}^2e^{-\varphi}dV_{\omega, Z}
   \end{equation*} 
  holds. 
\end{thm}

\begin{rem}
It is natural to expect that 
Theorem \ref{L2 extension 2} holds even if $Z$ has some mild singularities. 
But it seems to be unknown. 
$\hfill \Box$
\end{rem}

Theorem \ref{L2 extension 2} can be derived from Theorem $4$ of \cite{Ohs01} by a standard approximation technique. 
See also \cite{Kim10}. It seems most likely that a slight change of the proof of Theorem $4.2$ of \cite{Kim10} can yield Theorem $1.4$.
At any rate we give a self-contained proof in section $5$, as a courtesy to the reader. Theorem \ref{L2 extension 1}  in the present paper corresponds to Theorem $4$ in \cite{Ohs01} (but the situations in the two theorems are slight different). 
Theorem \ref{L2 extension 2} is used in the two critical steps.  
One step is to show the regularity of the restricted equilibrium weight (Theorem \ref{regularity}) 
and the other is to show a lower bound of restricted Bergman kernels (Theorem \ref{fundamental inequality for equilibrium}). 
Second, we only have a weak lower bound in the restricted case since it becomes harder  
to estimate the lower bound of the Bergman kernels precisely as \cite{Ber09}.  
We avoid this difficulty by using a proof of the restricted version of the Fujita-type approximation theorem. 
Note that a part of this strategy already appeared in \cite{BB10}. 
We elaborate this strategy using a weak lower bound 
and the comparison theorem for the Monge-Amp\`{e}re operator. 
From this, one can first get an integral representation of the restricted volume 
and then deduce Theorem \ref{Main Theorem}. 
Compared with \cite{Ber09} in the case $Z=X$, 
our proof of Theorem \ref{Main Theorem} is rather geometric 
thanks to the Fujita-type approximation. 
On the other hand, the convergence result obtained in this paper is weaker than that of \cite{Ber09}. 
It seems to be unknown whether convergence in a strict sense holds in the restricted case.  

\noindent

%２ MA operator%
\section{Monge-Amp\`{e}re operator}

We briefly review the definition of the Monge-Amp\`{e}re operator in this section. 
Fix a closed real smooth $(1,1)$-form $\theta$ defined on $X$. 
An $L^1_{\mathrm{loc}}$-function $ \psi $ in $X$ is called 
{\em $\theta$-plurisubharmonic}\ ($\theta$-psh for short) 
when the associated current 
$ \theta + dd^c \psi $ is positive (in the sense of currents). 
A function which is $\theta$-psh for some $\theta$ is called 
{\em quasi-plurisubharmonic}\ (quasi-psh for short). 
It is known that $\psi$ automatically becomes upper-semicontinuous by this condition. 
We denote the set of all $\theta$-psh functions by $\PSH(X,\theta)$. 
In this paper we are mainly interested in $\theta$ defined as 
$\theta = -dd^c \log h_L$, but this notion is in fact valid for an arbitrary $\theta$. 

Let $n$ be the dimension of $X$. The Monge-Amp\`{e}re operator should be defined as:  
\begin{equation*}
   \psi \mapsto  \MA(\psi) := (\theta + dd^c \psi)^n, 
\end{equation*}
but for general $\psi$, this is nonsense. 
The celebrated result of Bedford-Taylor (\cite{BT76}) tells us 
that the right hand side can be defined as a current 
for $\psi$ at least in the class $L^{\infty}_{\mathrm{loc}} \cap \PSH(X, \theta)$. 
That is, by induction on the exponent $q = 1,2,...,n$, it can be defined as: 
\begin{equation*}
   \int_X (\theta + dd^c \psi)^q \wedge \eta :=  
   \int_X (\theta + dd^c \psi)^{q-1} \wedge (\tau + \psi) dd^c \eta   
\end{equation*}
for each test form 
$ \eta \in C_0^{\infty}(X, \pigwedge^{n-q,n-q} T_X^*)$.  
Here $ \int_X $ denotes the canonical pairing of currents  and test forms, 
and $\tau$ denotes a local $dd^c$-potential of $\theta$. 
This is indeed well-defined and defines a closed positive current, 
because $ \tau + \psi $ is a bounded Borel function 
and $ (\theta + dd^c \psi)^{q-1} $ has measure coefficients by the induction hypothesis 
and by the fact that any closed positive current has measure coefficients. 
Bedford-Taylor's Monge-Amp\`{e}re products have useful continuity properties:  
\begin{prop}\label{continuity properties of BT}
  \begin{equation*}
  ( \theta + dd^c \psi_k )^n \to ( \theta + dd^c \psi ) 
  \ \ \ \ \ \text{ in the sense of currents }  
  \end{equation*} 
  for any sequence of $\theta$-psh functions which satisfies one of the following conditions. 
    \begin{itemize}
    \setlength{\itemsep}{0pt}
     \item[$(1)$]
        $\psi_k \searrow \psi$ pointwise in $X$.    
     \item[$(2)$]
        $\psi_k \nearrow \psi$ for almost every point in $X$. 
     \item[$(3)$] 
        $\psi_k \to \psi$ uniformly in any compact subset of $X$. 
   \end{itemize}
\end{prop}

It is still necessary to consider unbounded $\theta$-psh functions. 
On the other hand, for our purpose to investigate asymptotic behaviors of Bergman kernels, 
it is sufficient to deal with some special class of unbounded $\theta$-psh functions 
and we can omit a part of the contribution of unbounded loci. 
\begin{dfn}\label{definition of small unbounded locus}
 A $\theta$-psh function $\psi$ is said to have {\em a small unbounded locus} 
 if the pluripolar set $\psi^{-1}(- \infty)$ is contained 
 in some closed proper algebraic subset $S \subsetneq X$.  
 $\hfill \Box$
\end{dfn}
A quasi-psh function $\psi$ on $X$ 
is said to have {\em algebraic singularities}, 
if it can be locally written as 
\begin{equation}\label{algebraic singularity}
  \psi = c \cdot \log (\abs{f_1}^2 + ... + \abs{f_N}^2) + u 
\end{equation}
for some $c \in \mathbb{Q}_{\geq 0}$, 
non-zero regular functions $f_i$ $( 1 \leq i \leq N )$, 
and a smooth function $u$. 
Every $\theta$-psh function with algebraic singularities has a small unbounded locus. 
If we assume the subvariety $Z$ is smooth,  
$\iota^* \varphi + m^{-1}\log B_{X|Z}(m\varphi) $ gives 
the typical example of $\iota^*\theta$-psh function with algebraic singularities. 
\begin{dfn}\label{definition of MA}
 For a $\theta$-psh function $\psi$ on $X$ with a small unbounded locus, 
 $\MA(\psi)$ is defined to be
 \begin{equation*}
 \big\langle (\theta + dd^c \psi)^n  \big\rangle := 
 \text{ the zero extension of } ( \theta + dd^c \psi)^n . 
 \end{equation*}
 Note that the coefficient of $( \theta + dd^c \psi )^n $ is well-defined 
 as a measure on $X \setminus S$. 
 $\hfill \Box$
\end{dfn}
$\MA(\psi)$ actually defines a closed positive current on $X$ 
by famous Skoda's extension theorem. 
In particular, it has a finite mass on $X$. 
For a proof, see \cite{BEGZ08}, section 1. 
\begin{rem}
In that paper, the {\em non-pluripolar} Monge-Amp\`{e}re product was defined 
in fact for general $\theta$-psh functions on a compact K\"{a}hler manifold.
Note that this choice of ways to define the Monge-Amp\`{e}re operator makes 
$\MA(\psi)$ to have no mass on any pluripolar set 
so ignores some of the singularities of $\psi$. 
For this reason, $\langle (\theta + dd^c \psi)^n \rangle$ no longer has 
continuity property with respect to $\psi$. 
$\hfill \Box$
\end{rem}
We recall the fundamental fact established in \cite{BEGZ08} 
which states that 
the less singular $\theta$-psh function has the larger Monge-Amp\`{e}re mass. 
Recall that given two $\theta$-psh $\psi$ and $\psi'$, 
$\psi$ is said to be less singular than $\psi'$ 
if there exists a constant $C > 0$ such that $\psi' \leq \psi + C$ in $X$. 
We say that a $\theta$-psh function is {\em minimal singular} 
if it is minimal with respect to this partial order. 
When $\psi$ is less singular than $\psi'$ and $\psi'$ is less singular than $\psi$, 
we say that the two functions are equivalent with respect to singularities. 
This defines a equivalence relation in $\PSH(X, \theta)$. 
When $\theta \in c_1(L)$, 
any minimal singular $\theta$-psh function $\psi$ has a small unbounded locus. 
In fact, 
$\psi^{-1}(- \infty) \subseteq \mathbb{B}_+(L) $
holds. 
\begin{thm}[\cite{BEGZ08}, Theorem 1.16]\label{comparison theorem} 
If $\psi, \psi'$ are $\theta$-psh functions 
with small unbounded loci such that
$\psi$ is less singular than $\psi'$, then 
\begin{equation*}
\int_X \MA(\psi')
\leq 
\int_X \MA(\psi)
\end{equation*}
holds. 
\end{thm}
\begin{rem}
  It is unknown that Theorem \ref{comparison theorem} holds for general $\theta$-psh functions. 
  $\hfill \Box$
\end{rem}
The notion of types of $\theta$-psh functions with respect to singularities, 
explained in this subsection, 
are in fact determined by the closed positive currents $ T:= \theta + dd^c \psi $. 
Namely, a closed positive $(1,1)$-current $T \in \alpha$ is said to have 
a small unbounded locus if it can be written: $T = \theta + dd^c \psi$ 
with some $\psi$ which has a small bounded locus. 
Closed positive $(1,1)$-currents with algebraic singularities 
and those with minimal singularities can be defined in the same manner. 

%３ベルグマン核の漸近挙動%
\section{Restricted Bergman kernel asymptotics}

%３．１平衡なウェイト%
\subsection{Restricted equilibrium weight}

In this subsection, we introduce the notion of the restricted equilibrium weight
and discuss its properties, which we will use later 
to study asymptotics of restricted Bergman kernels. 
Unless otherwise stated, we fix a big line bundle $L$ on a smooth projective variety $X$ 
and a smooth metric $ h_L $. Let $\theta:=-dd^c \log h_L$ be the Chen curvature form. 
Given a subvariety $Z$ of $X$, 
there exists a canonical way to associate any smooth function 
to the $\theta$-psh function on $Z$. 
\begin{dfn}\label{definition of equilibrium weight}
For a smooth weight $\varphi \in C^{\infty}(X;\mathbb{R})$ and a subvariety $Z \subseteq X$,
the restricted equilibrium weight $P_{X|Z} \varphi$ is a function on $Z$ defined as follows:  
 \begin{equation}\label{definition of equilibrium}
   P_{X|Z} \varphi (z)  := 
     \soup \Bigg\{ 
             \iota^*\psi (z) \ \Bigg| 
             \begin{matrix} \ \psi \in \PSH(X, \theta) \\
                            \ \text{ with } \iota^*\psi \leq \iota^*\varphi \ \text{ on } \ Z \\ 
             \end{matrix}
           \Bigg\} 
 \end{equation}
for $ z \in Z$. Here $\iota:Z \hookrightarrow X$ denotes the inclusion map.
If there is no $\psi$ as above, $P_{X|Z} \varphi \equiv -\infty$ by definition. 
$\hfill \Box$
\end{dfn}
In the special case when $Z=X$,  
we use the notation $P_X$ instead of $P_{X|X}$ as in \cite{BB10}.
The symbol $\soup$ appeared in the above definition means
\begin{equation*}
 \soup_{\alpha} f_{\alpha}(z) := \limsup_{w \to z}\big(\sup_{\alpha} f_{\alpha}(w)\big)
\end{equation*}  
which is called {\em the regularized upper envelope} 
for a family of functions $\{ f_{\alpha} \}_{\alpha}$.
It is easily seen that $\iota^*P_X \varphi \leq P_{X|Z} \varphi \leq P_Z  \iota^*\varphi$ holds. 
By a classical result of Choquet (see e.g. \cite{Kli91}, Lemma 2.3.4.) and by the definition of $P_{X|Z} \varphi$, 
we get the following. 
\begin{lem}\label{Choquet lemma}
   Assume that $P_{X|Z}\varphi$ is not identically infinity on $Z$. Then there exists a countable non-decreasing family of $\theta$-psh functions 
   $\{ \psi_k \}_k$ $(k=1,2,...)$ 
   such that $\iota^*\psi_k \nearrow P_{X|Z} \varphi$ a.e., 
   otherwise $P_{X|Z}\varphi \equiv -\infty$.   
   In particular, $P_{X|Z} \varphi \in \PSH(Z, \iota^*\theta)$  
   unless $P_{X|Z}\varphi \equiv -\infty$. 
\end{lem}
Now assume that $Z$ is smooth and that $\iota(Z) \nsubseteq \mathbb{B}_+(L)$.
Then $P_{X|Z} \varphi$ has a small unbounded locus contained in $\iota^{-1}(\mathbb{B}_+(L))$.  
Then it follows that the Monge-Amp\`{e}re mass of $P_{X|Z}\varphi$ can be defined as:
\begin{equation*}
 \int_Z \MA(P_{X|Z} \varphi) := \int_Z \big\langle (\iota^*\theta + dd^c P_{X|Z} \varphi)^p \big\rangle
  =\int_{Z \setminus \iota^{-1}(\mathbb{B}_+(L))} (\iota^*\theta + dd^c P_{X|Z} \varphi)^p .
\end{equation*}  
The following is a consequence of Theorem \ref{comparison theorem}, and it enables us to 
substitute $\MA(\iota^*P_X \varphi)$ for $\MA(P_{X|Z} \varphi)$ to estimate the lower bound of the restricted Bergman kernels. 
This is a starting point of our strategy to prove Theorem \ref{Main Theorem}. 
\begin{thm}\label{comparison for equilibrium} 
Assume that $Z$ is smooth and that $\iota(Z) \nsubseteq \mathbb{B}_+(L)$. Then It holds that 
\begin{equation*}
\int_Z \MA(P_{X|Z} \varphi) = \int_Z \MA(\iota^*P_X \varphi). 
\end{equation*} 
\end{thm}
\begin{proof}
By Theorem \ref{comparison theorem}, we only have to show the first equality. 
The one side inequality $\geq$ is also an immediate consequence of Theorem \ref{comparison theorem}.
Since we may take $\psi_k$ minimal singular in Lemma \ref{Choquet lemma} 
(by exchanging $\psi_k$ by $\max\{\psi_k, P_X \varphi \}$),  
 \begin{equation*}
   \int_Z \MA(\iota^*P_X\varphi) = \int_Z \MA(\iota^*\psi_k)
 \end{equation*}
holds by Theorem \ref{comparison theorem}. 
On the other hand, since 
 \begin{equation*}
  (\iota^*\theta + dd^c \iota^*\psi_k)^p 
     \to (\iota^*\theta + dd^c P_{X|Z}\varphi)^p 
    \ \ \text{ on }  \ Z \setminus \iota^{-1}(\mathbb{B}_+(L))
 \end{equation*}
by the continuity property of the Monge-Amp\`{e}re operator, we have 
 \begin{equation*}
   \liminf_{k \to \infty} \int_Z \MA(\iota^*\psi_k) \geq \int_Z \MA(P_{X|Z}\varphi). 
 \end{equation*}
Therefore 
 \begin{equation*}
   \int_Z \MA(\iota^*P_X\varphi) \geq \int_Z \MA(P_{X|Z}\varphi). 
 \end{equation*}
\end{proof}

The next theorem is a key ingredient to represent
$\MA(P_{X|Z} \varphi)$ explicitly by $\varphi$. 
It states that the gradient of $P_{X|Z}\varphi$ is locally Lipschitz on $Z \setminus \mathbb{B}_+(L)$.  
\begin{thm}\label{regularity}
  Assume that $Z$ is smooth 
  and $\iota(Z) \nsubseteq \mathbb{B}_+(L)$. 
  Then $P_{X|Z}\varphi$ has Lipschitz continuous first derivatives 
  outside of $\iota^{-1}(\mathbb{B}_+(L))$. 
  Namely, 
  \begin{equation*}
   P_{X|Z} \varphi \in C^{1,1}\big(Z \setminus \iota^{-1}(\mathbb{B}_+(L))\big).
  \end{equation*}
  Moreover, 
  \begin{equation*}
   (\iota^*\theta + dd^c P_{X|Z} \varphi )^p = (\iota^*\theta + dd^c \iota^*\varphi )^p 
   \ \ \ \ \  in \ the \ set \ \ \ \{ P_{X|Z} \varphi = \iota^*\varphi \} \setminus \iota^{-1}(\mathbb{B}_+(L))
  \end{equation*}
  a.e.\ with \  respect \  to \ $ d \mu $. 
\end{thm}
\begin{proof} 
The proof is almost the same as $Z=X$ case in \cite{Ber09} 
except that in the restricted case we need  the Ohsawa-Takegoshi-type $L^2$-extension theorem 
for an arbitrary smooth subvariety (Theorem \ref{L2 extension 2}). 
We sketch the proof and omit the detail. 

Let $Y$ be the total space of the dual line bundle $L^*$, identifying the base $X$ with 
its embedding as the zero-section in $Y$, and  $\pi : Y \to X$ be the projection map. 
Given $\psi \in \PSH(X, \theta)$, one can associate a psh function $\chi_{\psi}$ defined on $Y$, as follows: 
\begin{equation*}
  \chi_{\psi}(x,w) := \log \abs{w}^2_{h_L^{-1}} + \psi(x) \ \ \ \ \ (x \in X, w \in L_x). 
\end{equation*}

Berman's original argument is modeled on the proof of 
Bedford-Taylor for $C^{1,1}$-regularity of the solution of the Dirichlet problem 
for the complex Monge-Amp\`{e}re equation in the unit-ball in $\mathbb{C}^n$. 
As opposed to the unit ball, $X$ has no global holomorphic vector fields. 
But one can reduce the regularity problem of $P_X \varphi$ on $X$ to a problem of $\chi_{P_X \varphi}$ on $Y$, 
where enough many vector fields exist. 
This argument is still valid in the restricted case once one can construct the suitable vector fields on $\pi^{-1}(Z)$ extended to $Y$. 

For a proof, it is enough to show the regularity of $\chi_{P_X \varphi}$ at 
any given point $y_0 \in \pi^{-1} (Z \setminus \mathbb{B}_+(L) ) \setminus Z$. 
By Kodaira's lemma, there exists an effective divisor $E$ on $X$ such that 
$y_0 \notin \pi^{-1}(\Supp E)$ and $mL = A + E$ hold 
with some positive integer $m$ and ample $\mathbb{Z}$-divisor $A$.  
We may assume $m=1$ for the proof of the Theorem \ref{regularity} 
since $mP_{X|Z}\varphi = P_{X|Z}(m\varphi)$ holds. 
By this decomposition, 
we can construct some $\psi_0= \psi_A + \psi_E$, $\theta_L =\theta_A + \theta_E$ 
such that 
$\theta_A + dd^c\psi_A > 0$ is smooth and $\theta_E + dd^c\psi_E \geq 0$ 
has singularities only on $E$. 
Indeed, it is enough to set 
$\theta_E := -dd^c \log h_{E,\alpha} $, $\psi_E := \log \abs{f_{\alpha}}^2 + \log h_{E,\alpha}$ 
for some smooth metric $h_E$ and system of local equations $\{ f_{\alpha} \}$. 

\begin{lem}\label{construction of vector fields} 
  There exist  holomorphic vector fields 
  $V_1,...,V_{p+1}$ on $\pi^{-1}(Z)$ satisfying the following properties. 
   \begin{itemize}
    \setlength{\itemsep}{0pt}
     \item[$(1)$]
        $V_1,...,V_{p+1}$ is linearly independent at $y_0$. 
     \item[$(2)$]
        There exist holomorphic vector fields 
        $\widetilde{V}_1,...,\widetilde{V}_{p+1}$ on $Y$ 
        such that 
        $\widetilde{V}_i|_{\pi^{-1}(Z)} = V_i$ $(1 \leq i \leq p+1)$.    
     \item[$(3)$] 
        For any fixed $k \in \mathbb{N}$, $\widetilde{V}_i $ $(1 \leq i \leq p+1)$ can be chosen to have zeros of order at least $k$ along $X$ 
        and $\pi^{-1}(\Supp E)$. 
        To be precise, 
      \begin{equation*}
                  \abs{\widetilde{V}_i} \leq C(k) \cdot \abs{w}^k, \ \  
                  \abs{\widetilde{V}_i} \leq C(k) \cdot \abs{f_{\alpha}(z)}^k  
      \end{equation*} 
        hold locally in the set $\{ \chi_{\psi_0} \leq 1 \}$ for some constant $C(k)$ depending on $k$. 
   \end{itemize}
\end{lem}
\begin{proof}
  Let $\widehat{Y} := \mathbb{P}(\mathcal{O}(-L)\oplus\mathcal{O}))$ 
  be the Zariski closure of $Y$. 
  Consider the line bundle 
  $\pi^*L^{k_0} \otimes H_{\mathbb{P}(\mathcal{O}(-L)\oplus\mathcal{O})}$ 
  on $\widehat{Y}$ and its metric 
 \begin{equation*}  
    h_{{k_0},\alpha} := \pi^* h_{L,\alpha}^{k_0} + \log (1+ e^{\chi_{\varphi}}) 
 \end{equation*}
  with weight 
 \begin{equation*}
    \psi_{k_0} := \pi^*\big( k_0 (\psi_A + (1+k_0^{-1/2})\psi_E) \big). 
 \end{equation*} 
 Here $H_{\mathbb{P}(\mathcal{O}(-L)\oplus\mathcal{O})}$ denotes 
 the fiberwise hyperplane bundle. 
 For $w$-direction, $\log(1+e^{\chi_{\varphi}})$ has the strictly positive curvature and 
 for $x$-direction, $\psi_A + (1+k_0^{-1/2})\psi_E$ is $\theta_L$-strictly positive 
 if we take $k_0$ sufficiently large. 
 Thus $h_{k_0}e^{-\psi_{k_0}}$ has strictly positive curvature in $\widehat{Y}$. 
 From this, taking sufficiently large $k_1$, 
 we can use Theorem \ref{L2 extension 2} 
 to get holomorphic sections  
 \begin{equation*}
  V_1,...,V_{p+1} \in H^0(\widehat{\pi^{-1}(Z)}, 
                            \mathcal{O}(T'_{\widehat{\pi^{-1}(Z)}} \otimes (\pi^*L^{k_0}\otimes H_{\mathbb{P}(\mathcal{O}(-L)\oplus\mathcal{O})})^{k_1})) 
 \end{equation*}
 which correspond to some basis $V_{1,0},...,V_{p+1,0}$ of $T'_{\pi^{-1}(Z), y_0}$. 
 If we use Theorem \ref{L2 extension 2} once more and take further large $k_1$, 
 it can be seen that 
  $V_1,...,V_{p+1}$ are restrictions of some 
 \begin{equation*}
  \widetilde{V}_1,...,\widetilde{V}_{p+1} \in H^0(\widehat{Y}, 
                            \mathcal{O}(T'_{\widehat{Y}} \otimes (\pi^*L^{k_0}\otimes H_{\mathbb{P}(\mathcal{O}(-L)\oplus\mathcal{O})})^{k_1})) 
 \end{equation*} 
 which are integrable with respect to $(h_{k_0}e^{-\psi_{k_0}})^{k_1}$. 
 Note that 
 $H_{\mathbb{P}(\mathcal{O}(-L)\oplus\mathcal{O})}|_{Y}$ is trivial 
 and that 
 $\pi^*L = -[X]$ (the dual of the line bundle defined by the divisor $X \subseteq Y$). 
 Therefore $\widetilde{V}_1,...,\widetilde{V}_{p+1}$ can be identified 
 with holomorphic vector fields over $Y$ 
 having zeros of order at least $k_0k_1$ along $X$. 
 Further, by the integrability condition, we get
 \begin{equation*}
  \abs{\widetilde{V}_i(x,w_{\alpha})} \leq C \cdot 
                         \bigg(  \abs{w_{\alpha}}^{k_0} \cdot \abs{f_{\alpha}(x)}^{k_0(1+k_0^{-1/2})} \cdot \abs{w_{\alpha}} \bigg)^{k_1}
 \end{equation*}
 hence 
 \begin{equation*}
  \abs{\widetilde{V}_i(x,w_{\alpha})} \leq C(k) \cdot 
                         \bigg( \abs{w_{\alpha}} \cdot \abs{f_{\alpha}(x)} \bigg)^{(k_0 + 1)k_1} \cdot \abs{f_{\alpha}(x)}^k. 
 \end{equation*} 
 The boundedness of $\abs{w_{\alpha}} \cdot \abs{f_{\alpha}(x)}$ in $\{ \chi_{\psi_0} \leq 1 \}$ implies the conclusion. 
\end{proof}
Actually, Lemma \ref{construction of vector fields} assures the existence of desired vector fields $V_i$ ($1 \leq i \leq p$) 
and one can repeat the proof of Theorem 3.4 in \cite{Ber09}. 
\end{proof}
On the other hand, the repeating the proof of the $Z=X$ case in Proposition 3.1 of \cite{Ber09} 
gives the following. 
\begin{lem}
   In the situation of Theorem \ref{regularity}, 
 \begin{itemize}\label{lemma for representation formula}
  \setlength{\itemsep}{0pt}
   \item[$(1)$]
     \begin{equation*}
          P_{X|Z} \varphi = \iota^* \varphi \ \ a.e.\ with \ respect \ to \ \MA(P_{X|Z} \varphi).
     \end{equation*}
   \item[$(2)$]
     \begin{equation*}
          P_{X|Z} \varphi(z_0) = \iota^* \varphi(z_0) \ 
           \Rightarrow \ (\iota^*\theta + dd^c \iota^* \varphi)(z_0) \geq 0.
     \end{equation*}
 \end{itemize}
\end{lem} 
  As a consequence of Theorem \ref{regularity} 
  and Lemma \ref{lemma for representation formula} $(1)$, 
  we obtained the desired representation formula for  $\MA(P_{X|Z} \varphi)$ as in the case $Z=X$. 
\begin{thm}\label{representation formula for equilibrium}
  Assume that $Z$ is smooth 
  and $\iota(Z) \nsubseteq \mathbb{B}_+(L)$. 
  Then the identity 
 \begin{equation}\label{representation formula}
  \MA(P_{X|Z} \varphi )= \mathbf{1}_{\{ P_{X|Z} \varphi 
                   =\iota^* \varphi \}} \cdot ( \iota^* \theta + dd^c \iota^* \varphi )^p
 \end{equation}
 holds. 
 Here $ \mathbf{1}_{\{P_{X|Z} \varphi = \iota^* \varphi \}} $ denotes 
 the characteristic function of the set $\{P_{X|Z} \varphi = \iota^* \varphi \}$. 
 In particular, the measure $\MA(P_{X|Z} \varphi )$ has $L^{\infty}$-density 
 with respect to $d \mu$. 
\end{thm}
%

%３．２ベルグマン核の漸近挙動%
\subsection{Restricted Bergman kernel asymptotics}

From now on, we compare $P_{X|Z} \varphi $ with $B_{X|Z}(m \varphi)$ in detail. 
Fix notations as in the previous subsections. 
In this subsection we always assume that 
$Z$ is a smooth subvariety of $X$ and that $\iota(Z) \nsubseteq \mathbb{B}_+(L)$ holds. 
First we specify the upper bound of restricted Bergman kernels and show the half of our main result. 
\begin{prop}\label{upper bound 4} 
 \begin{equation*}
  \limsup_{m \to \infty} \frac{B_{X|Z}(m \varphi)}{m^p/p!} d \mu
   \leq \MA(P_{X|Z} \varphi).
 \end{equation*}
\end{prop}
\begin{proof}
This is deduced from the two estimates 
about the upper bound of Bergman kernels. 
First, we show the so-called {\em ``Berman's local holomorphic Morse inequality''} 
(see \cite{Ber04}, Theorem 1.1) in the restricted case. 
The proof in the case $Z=X$ is applicable with no change. 

{\em Claim $(1)$}:
\begin{equation}\label{upper bound 1}
  \limsup_{m \to \infty} \frac{B_{X|Z}(m \varphi)}{m^p/p!}d \mu \leq 
  \mathbf{1}_{\{(\iota^* \theta + dd^c \iota^* \varphi ) \geq 0\}} 
  \cdot ( \iota^*\theta + dd^c \iota^* \varphi )^p.
\end{equation}

{\em Proof of the claim $(1)$.}
Fix any $z_0 \in Z$. If we take an appropriate trivialization patch $U$ around $z_0$ with 
$h_{L, U}(z_0)e^{-\varphi(z_0)}=1$ and denote the eigenvalues of 
$ \iota^*\theta + dd^c \iota^* \varphi $ with respect to the form 
$\frac{\sqrt{-1}}{2}\sum_{i=1}^{p} dz_i \wedge d\overline{z}_i $ 
at $z_0$ by $ \lambda_1,...,\lambda_p$, then 
for an arbitrary section $s \in H^0(X|Z, \mathcal{O}(mL))$ 
with $\norm{s}^2_{m \varphi} = 1$, we have
\begin{equation*}
\begin{split}
  & \frac{\abs{s(z_0)}^2_{m \varphi}}{m^p/p!} 
  = \frac{\abs{s_{U}(z_0)}^2}{m^p/p!}\\
  &\leq \Bigg( \int_{\abs{z}\leq\frac{\log{m}}{\sqrt{m}}}
               \abs{s_{U}}^2e^{-m\sum \lambda_i\abs{z_i}^2 }d\lambda(z) \Bigg) 
        \Bigg( \int_{\abs{z}\leq\frac{\log{m}}{\sqrt{m}}}
               e^{-m\sum \lambda_i\abs{z_i}^2} d\lambda(z) \cdot m^p/p! \Bigg)^{-1}
\end{split}
\end{equation*}
by the mean value inequality for subharmonic functions. 
Here $ d \lambda $ denotes the Lebesgue measure with respect to $z_i$. 
The $\limsup_{m \to \infty}$ of the numerator in the last side 
is bounded by $\det_{d \mu} d \lambda (z_0)$, 
and the denominator behaves as follows if we let $m \to \infty$:
\begin{equation*}
 \frac{1}{p!}\int_{\abs{w} \leq \log{m}} e^{- \sum \lambda_i \abs{w_i}^2} d \lambda(w) \to
  \begin{cases}
     \pi^p \big/ ({p! \lambda_1 \lambda_2 \cdots \lambda_p}) & \text{ if $\lambda_i \geq 0$}\\
     \infty  & \text{otherwise.}
  \end{cases}
\end{equation*}
(Here we use $w=\sqrt{m}z$ as a new variable.)
From this, one can deduce the claim. 

The second claim is a direct consequence of the definition of $P_{X|Z} \varphi$, 
and motivates the definition as well. 

{\em Claim $(2)$}: 
\begin{equation}\label{upper bound 2}
  \frac{B_{X|Z}(m \varphi)}{m^p/p!} \leq 
   e^{-m(\iota^* \varphi - P_{X|Z} \varphi)} \cdot \sup_Z \frac{B_{X|Z}(m \varphi)}{m^p/p!}.
\end{equation}
  
{\em Proof of the claim $(2)$.}
Note that the supremum in the right hand side is finite by claim $(1)$. 
Fix any $z_0 \in Z$ and take any $s \in H^0(X|Z, \mathcal{O}(mL))$ satisfying 
$\abs{s(z_0)}^2_{m \varphi}=B_{X|Z}(m \varphi)(z_0)$ and $\norm{s}^2_{m \varphi}=1$.
Since $\abs{s(z)}^2_{m \varphi} \leq \sup_Z B_{X|Z}(m \varphi) $ 
for any $z \in Z$, we have 
\begin{equation*}
 \frac{1}{m}\big( \log \abs{s(z)}^2_{h^m_L} - \log \sup_Z B_{X|Z}(m \varphi)\big) \leq 
  \iota^* \varphi \text{\ \ in $Z$}.
\end{equation*} 
Since the left hand side is the pull-back of a $\theta$-psh function on $X$  
the above inequality implies 
\begin{equation*}
 \frac{1}{m}\big( \log \abs{s(z)}^2_{h^m_L} - \log \sup_Z B_{X|Z}(m \varphi)\big) \leq 
  P_{X|Z} \varphi \text{\ \ in $Z$}.
\end{equation*} 
Thus the claim $(2)$ is obtained. 

Proposition \ref{upper bound 4} is now easily proved. 
Actually, claim $(1)$ and Lemma \ref{lemma for representation formula} $(2)$ imply  
\begin{equation*}
   \limsup_{m \to \infty} \frac{B_{X|Z}(m \varphi)}{m^p/p!} d \mu
   \leq ( \iota^*\theta + dd^c \iota^* \varphi )^p 
   \text{ \ \ in $\{ P_{X|Z} \varphi = \iota^* \varphi \}$}
\end{equation*}
and claim $(2)$ implies the pointwise convergence  
\begin{equation}\label{upper bound 3}
 \frac{B_{X|Z}(m \varphi)}{m^p/p!} \to 0 
  \text{ \ \ ($m \to \infty$) \ in $\{ P_{X|Z} \varphi \neq \iota^* \varphi \}$}
\end{equation}
so one can conclude Proposition \ref{upper bound 4} 
by Theorem \ref{representation formula for equilibrium}.
\end{proof}
\begin{cor}\label{half of Fujita}
  \begin{equation*}
    \Vol_{X|Z}(L) \leq \int_Z \MA(P_{X|Z}\varphi)
  \end{equation*}
\end{cor}
\begin{proof}
Since $\MA(P_{X|Z}\varphi)$ has $L^{\infty}$-density 
by Theorem \ref{representation formula for equilibrium}, 
we can apply Fatou's lemma to (\ref{upper bound 4}). 
\end{proof}

We can now derive the fundamental relation
between $P_{X|Z} \varphi$ and $B_{X|Z}(m \varphi)$.
\begin{thm}\label{fundamental inequality for equilibrium}
For every compact set $K \Subset Z \setminus \iota^{-1}( \mathbb{B}_+(L))$, 
there exist an integer $m_0$ and a positive constant $C \geq 0$ such that the inequality
\begin{equation}\label{fundamental inequality}
C^{-1} \cdot e^{-m(\iota^*\varphi-P_{X|Z} \varphi)} \leq B_{X|Z}(m \varphi) 
\leq C \cdot m^p e^{-m(\iota^*\varphi-P_{X|Z}\varphi)}
\end{equation}
holds.
\end{thm}
\begin{proof}
The right hand side inequality is a direct consequence of Proposition \ref{upper bound 4} 
and Theorem \ref{representation formula for equilibrium}. 
We will show the left hand side. 
By the extremal property of the Bergman kernel, it is enough to show the following claim.  

{\em Claim}:
There exist some $m_0$ , $C$ and section
$s_m \in H^0(X|Z, \mathcal{O}(mL))$ 
for each $m \geq m_0$ such that
\begin{itemize}
 \setlength{\itemsep}{0pt}
  \item[$(1)$]
    $ \abs{s_m(z)}^2_{m\psi_k} \geq 
      C^{-1}  \ \ for \ any \ z \in K, \ k \in \mathbb{N} $,  
  \item[$(2)$]
    $ \norm{s_m}^2_{m \varphi} \leq C $. 
\end{itemize} 
Here $\psi_k \in \PSH(X, \theta)$ are taken to satisfy 
$\iota^*\psi_k \nearrow P_{X|Z} \varphi \ $ a.e.\ with respect to $d \mu$. 
Actually, this implies
\begin{equation*}
 B_{X|Z}(m \varphi) 
      \geq \frac{\abs{s_m(z)}^2_{m \varphi}}{\norm{s_m}^2_{m \varphi}} 
      \geq C^{-2}e^{-m(\iota^*\varphi - \iota^*\psi_k)}
\end{equation*}
so letting $k \to \infty$ we get the inequality. 

{\em Proof of the claim.}
Fix $z \in K$. 
By Kodaira's lemma, we may take some ample $\mathbb{Q}$-divisor $A$ 
and some effective $\mathbb{Q}$-divisor $E$ on $X$ 
satisfying $L=A+E$. 
From this decomposition, we may construct a $\theta$-psh function $\psi_0$ 
with $\psi_0^{-1}(-\infty) \subseteq \Supp E$, $\psi_0 \leq \varphi$.  
Then using Theorem \ref{L2 extension 2} twice, 
we may find suitable $m_0, C$ and sections $s_m \in H^0(X|Z, \mathcal{O}(mL))$ 
for each $m \geq m_0$ such that
\begin{itemize}
 \setlength{\itemsep}{0pt}
  \item[$(1)$]
    $\abs{s_m(z)}^2_{\psi_{m,k}} = 1$
  \item[$(2)$]
    $\norm{s_m}^2_{\psi_{m,k}} \leq C$ , 
\end{itemize}
where $\psi_{m,k} =(m-m_0)\psi_k + m_0 \psi_0 $. 
Then we infer 
\begin{equation*}
 \norm{s_m}^2_{m \varphi} \leq \norm{s_m} ^2_{\psi_{m,k}} \leq C
\end{equation*}
and since we may assume $e^{m_0(\varphi-\psi_0)(z)} \leq C$ 
by the smoothness of $\psi_0$ around $z$, 
\begin{equation*}
 1 = \abs{s_m(z)}^2_ {\psi_{m,k}} \leq C\abs{s_m(z)}^2_{m \varphi_k} .
\end{equation*}
Here $C$ depends on $m_0$ and $K$. 
\end{proof}

As a consequence of the above results, 
the sequence of the Monge-Amp\`{e}re mass 
of the following Fubini-Study like potential functions 
converges to the Monge-Amp\`{e}re mass of the restricted equilibrium weight.
This fact corresponds to the description of restricted volumes  
via moving intersection numbers 
(see Theorem \ref{moving intersection number description}), 
and has a key role for us to prove the local version of the restricted Fujita approximation 
in the next subsection.
Let us define: 
 \begin{equation}\label{definition of FS}
  u_m:= \iota^*\varphi + \frac{1}{m} \log B_{X|Z}(m \varphi).
 \end{equation}
\begin{thm}\label{convergence of FS}
\begin{equation*}
   u_m \to P_{X|Z} \varphi  
   \ \ \ \ \text{  uniformly in any compact subset of \ $Z \setminus \iota^{-1}(\mathbb{B}_+(L))$}, 
\end{equation*}
and
\begin{equation*}
   \MA(u_m) \to \MA(P_{X|Z} \varphi)
   \ \ \ (m \to \infty)
\end{equation*}
in the sense of currents.
\end{thm}
\begin{proof}
The inequality (\ref{fundamental inequality}) is equivalent to 
\begin{equation*}
 - \frac{\log C}{m} + P_{X|Z} \varphi 
 \leq \iota^* \varphi + \frac{1}{m}\log B_{X|Z}(m \varphi) 
 \leq \frac{\log C+p\log m}{m} + P_{X|Z} \varphi.
\end{equation*}
This estimate implies that, on any compact subset of $Z \setminus \iota^{-1}(\mathbb{B}_+(L))$, 
$u_m$ converges uniformly to $P_{X|Z} \varphi$.
By the continuity property of the Monge-Amp\`{e}re operator, we deduce
\begin{equation*}
  (\iota^* \theta + dd^c u_m )^p \to (\iota^* \theta + dd^c P_{X|Z} \varphi)^p  
  \text{ \ \ in $Z \setminus \iota^{-1}(\mathbb{B}_+(L))$}.
\end{equation*}
In particular, 
\begin{equation*}
  \liminf_{m \to \infty} 
        \int_{Z \setminus \iota^{-1}(\mathbb{B}_+(L))}(\iota^* \theta + dd^c u_m)^p 
  \geq  \int_{Z \setminus \iota^{-1}(\mathbb{B}_+(L))}(\iota^* \theta + dd^c P_{X|Z} \varphi )^p
\end{equation*}
holds. Therefore we only have to show  
\begin{equation*}
  \limsup_{m \to \infty} 
        \int_{Z \setminus \iota^{-1}(\mathbb{B}_+(L))}(\iota^* \theta + dd^c u_m)^p 
  \leq  \int_{Z \setminus \iota^{-1}(\mathbb{B}_+(L))} (\iota^* \theta + dd^c P_{X|Z} \varphi )^p, 
\end{equation*}
because we already have the current convergence  
in $Z \setminus \iota^{-1}(\mathbb{B}_+(L))$, 
but this is directly seen by Theorem \ref{comparison theorem} 
and Theorem \ref{comparison for equilibrium}.
\end{proof}

%３.３Fujita近似はじめ%
\subsection{Restricted Fujita-type approximation}
In this subsection, we first give a proof of the restricted Fujita approximation theorem
and then finish the poof of Theorem \ref{Main Theorem}.  
\begin{thm}[\cite{Tak06}, Theorem 3.1, \cite{ELMNP09}, Theorem 2.13]\label{Fujita approximation}
Let $X$ be a smooth projective variety, $\iota:Z \hookrightarrow X$ a subvariety, 
and $L$ a big line bundle on $X$. 
Then for an arbitrary $\varepsilon > 0$, the following diagram is commutative, 
where $\pi_Z$, $\pi_X$ are modifications 
and $ \widetilde{Z}, \widetilde{X}$ are smooth such that
\begin{itemize}
 \setlength{\itemsep}{0pt}
  \item[$(1)$]
    in the sense of linear equivalence between $\mathbb{Q}$-divisors, 
    $\pi_X^* L =A+E$ holds for some semiample and big divisor $A$ and effective divisor $E$, 
    and
  \item[$(2)$]
    $\Vol_{\widetilde{X}|\widetilde{Z}}(A) \leq \Vol_{X|Z}(L) \leq \Vol_{\widetilde{X}|\widetilde{Z}}(A) + \varepsilon$  
\end{itemize} 
hold. 
\[\xymatrix{
    {\widetilde{Z}}\ar[r]^{\widetilde{\iota}}\ar[d]_{\pi_Z} 
  & {\widetilde{X}}\ar[d]^{\pi_X}\\ 
    {Z}\ar[r]_{\iota}
  &  X
}
\]
\end{thm}
\begin{rem} 
 By the continuity of the restricted volume 
 (see Theorem A in \cite{ELMNP09} or (\ref{restricted volume can be written by MA})), 
 the divisor $A$ in Theorem \ref{Fujita approximation} can be taken ample. 
  This is shown as follows. 
  First we get a decomposition of $\mathbb{Q}$-divisor: 
  $A=A_0+E_0=((1-\delta)A + \delta A_0 )  + (\delta E_0)$ by Kodaira's lemma. 
  Then $A_{\delta}:= (1-\delta)A + \delta A_0$ is ample and letting $\delta \to 0$, 
  $\Vol_{X|Z}(A_{\delta})$ approximates $\Vol_{X|Z}(A)$. 
  From this, it follows that 
  \begin{equation}\label{limsup is lim}
    \Vol_{X|Z}(L) 
     = \lim_{m \to \infty} \frac{\dim H^0(X|Z, \mathcal{O}({mL}))}{m^p/p!}
  \end{equation}
 holds. 
 Indeed one can reduce this to the case when $L$ is ample. 
 Since the Serre vanishing theorem forces 
 $H^0(X|Z, \mathcal{O}(mL)) = H^0(Z, \mathcal{O}(mL))$ in this case, 
 we may assume $Z=X$. 
 Then (\ref{limsup is lim}) is obtained from the Riemann-Roch theorem. 
 $\hfill \Box$
\end{rem}
Although a proof of Theorem \ref{Fujita approximation} 
is already obtained in \cite{Tak06} or \cite{ELMNP09}, 
we have to reprove this to show the local version 
(Theorem \ref{Main Theorem}) at the same time.
Our proof of Theorem \ref{Fujita approximation} is essentially the same 
as the proof in \cite{Tak06} or \cite{ELMNP09}, 
but we need a more direct proof and do not use a characterization of restricted volumes
via multiplier ideal sheaves. 
We need the following 
{\em ``The uniformly globally generation theorem''}, 
which was first proved in \cite{Siu98}. 
It can also be obtained as a corollary of Theorem \ref{L2 extension 2}.  
\begin{prop}[\cite{Siu98}, Proposition 1]\label{UGG}
 Given a smooth projective variety $X$, there exists a line bundle $G$ 
 such that for any pseudo-effective line bundle $F$ on $X$ 
 with a singular Hermitian metric $h_Fe^{-\psi}$  
 whose Chern curvature current is positive, 
 the sheaf $ \mathcal{O}(F+G) \otimes \mathcal{I}(\psi) $
 is globally generated. 
\end{prop}
We also need the following lemma which can be shown by simple algebraic computations. 
For a proof, see e.g. 2.2.C of \cite{Laz04}. 
\begin{lem}\label{lemma for Fujita approximation}
For an arbitrary line bundle $G$ on $X$ and a positive number $ \varepsilon > 0 $, there exist a subsequence
$\{\ell_k\}(k=1,2,...)$ and an integer $m_0$ such that
 \begin{equation*}
  \frac{\dim H^0(X|Z,\mathcal{O}(\ell_k(mL-G)))}{{\ell_k}^p/p!} 
  \geq m^p\big(\Vol_{X|Z}(L)-\varepsilon\big)
 \end{equation*}
for any $m \geq m_0$.
\end{lem}

\begin{proof}[Proof of Theorem \ref{Fujita approximation}.]

Throughout this proof, 
we fix some $G$ which appeared in the Theorem \ref{UGG}, 
and a smooth metric $h_G$ on $G$. 
For any fixed integer $m$, we define the weight of $h_L^mh_G^{-1}$ as follows:
\begin{equation*}
  u_m 
  := \varphi 
     + \frac{1}{m} \log (\abs{s_{m,1}}_{ m \varphi}^2+...+\abs{s_{m,N(m)}}_{m \varphi}^2), 
\end{equation*} 
where $ \{ s_{m,1},...,s_{m,N(m)} \} $ is a complete orthonormal system of
$H^0(X,\mathcal{O}(mL-G))$
with respect to the norm 

\begin{equation*}
\norm{s}_{m \varphi} ^2:= \int_X \abs{s}_{m \varphi}^2 d \mu
\end{equation*}
\begin{equation*}
\abs{s}_{m \varphi}^2 := (h_L^mh_G^{-1})(s,s)e^{-m \varphi}.
\end{equation*} 

This is essentially the same as $u_m$ in Theorem \ref{convergence of FS} 
($\iota=\mathrm{id}$ case).
In fact, as in subsection 3.2, we can get
\begin{equation}\label{T m convergence to T}
    \langle (\iota^*T_m)^p \rangle \to \langle (\iota^*T)^p \rangle, 
\end{equation}
where $T_m:=\theta+dd^cu_m$, $T:=\theta+dd^cP_X \varphi$  
assuming that $\iota$ is a closed embedding 
and that $Z$ is smooth, $\iota(Z) \nsubseteq \mathbb{B}_+(L)$. 
Here the difference caused by $G$ does not matter, 
for $G$ has no contribution to the asymptotic behavior of 
$H^0(X, \mathcal{O}(mL-G))$ thanks to the bigness of $L$. 

If we set $\mathcal{J} $ as the ideal sheaf 
generated locally by $s_{m,1},...,s_{m,N(m)}$, then 
$\mathcal{J} \subseteq \mathcal{I}(mu_m)$ holds. 
Therefore, by taking a log resolution we have the following commutative diagram, 
where $\pi_Z^*$ and $\pi_X^*$ are modifications from smooth projective varieties 
such that 
\begin{equation}\label{condition of modification 1}
  \pi_X^* \mathcal{I}(mu_m) 
  = \mathcal{O}_{\widetilde{X}}(-E'),\ \pi_X^* \mathcal{J} 
  = \mathcal{O}_{\widetilde{X}}(-F'), \ \text{ and } \  E' \leq F'.
\end{equation}
\[\xymatrix{
   {\widetilde{Z}}\ar[r]^{\widetilde{\iota}}\ar[d]_{\pi_Z} 
 & {\widetilde{X}}\ar[d]^{\pi_X}\\ 
   {Z}\ar[r]_{\iota} 
 & X
}
\]
Moreover, since $T_m$ has algebraic singularities, 
we may assume
\begin{equation}\label{condition of modification 2}
 \pi_X^*T_m = \gamma + [F], 
\end{equation}
where $\gamma $ is a smooth semipositive form and $F:=F'/m$. 
$[F]$ denotes the closed positive $(1,1)$-current defined by $F$. 
We claim that this diagram actually satisfies the condition 
in Theorem \ref{Fujita approximation} for a sufficiently large $m$. 

By Proposition \ref{UGG}, $\mathcal{O}(mL) \otimes \mathcal{I}(mu_m)$ is globally generated.
Therefore its pull-back $\mathcal{O}(m\pi_X^*L-E')$ is also globally generated. 
For this reason, we may have a semiample divisor $A'$ satisfying $ m \pi_X^* L = A' + E' $. 
Then the subadditivity property of multiplier ideal sheaves (see \cite{Laz04} 9.5.B) implies  
\begin{equation*}
\begin{split}
  &H^0(\widetilde{X}|\widetilde{Z},\mathcal{O}(\ell A')) 
   = H^0(\widetilde{X}|\widetilde{Z},\mathcal{O}(\ell(m\pi_X^*L-E'))) \\ 
 &= H^0(\widetilde{X}|\widetilde{Z},\pi_X^*(\mathcal{O}(\ell mL)\otimes\mathcal{I}(mu_m)^{\ell})) 
 \supseteq H^0(\widetilde{X}|\widetilde{Z},\pi_X^*(\mathcal{O}(\ell mL)\otimes\mathcal{I}(\ell mu_m)))
\end{split}
\end{equation*}
and we get 
\begin{equation*}
\begin{split}
 H^0(\widetilde{X}|\widetilde{Z},\pi_X^*(\mathcal{O}(\ell mL)\otimes\mathcal{I}(\ell mu_m)))  
 &\supseteq H^0(X|Z,\mathcal{O}(\ell mL)\otimes {\pi_X}_*\pi_X^* \mathcal{I}(\ell mu_m)) \\ 
 &=H^0(X|Z,\mathcal{O}(\ell mL) \otimes \mathcal{I}(\ell mu_m)) 
\end{split}
\end{equation*}
by the integral closedness of $\mathcal{I}(\ell mu_m)$. 
Further,  
\begin{equation*}
\begin{split} 
  H^0(X|Z,\mathcal{O}(\ell mL) \otimes \mathcal{I}(\ell mu_m)) 
 &\supseteq H^0(X|Z,\mathcal{O}(\ell(mL-G)) \otimes \mathcal{I}(\ell mu_m)) \\
 &=H^0(X|Z,\mathcal{O}(\ell (mL-G))) 
\end{split}
\end{equation*}
by the definition of $u_m$. 
Consequently, with Lemma \ref{lemma for Fujita approximation}, 
it can be seen that there exists a subsequence $\{ \ell_k \}$ and a sufficiently large $m$ 
such that
\begin{equation*}
 \frac{\dim H^0(\widetilde{X}|\widetilde{Z},\mathcal{O}(\ell_kA'))}{{{\ell}_k}^p/p!}
 \geq m^p(\Vol_{X|Z}(L)-\varepsilon). 
\end{equation*} 
Setting $A:=A'/m$, $E:=E'/m$, 
By homogeneity of restricted volume (\cite{ELMNP09} Lemma 2.2),  
\begin{equation*}
\Vol_{\widetilde{X}|\widetilde{Z}}(A) \geq \Vol_{X|Z}(L) - \varepsilon
\end{equation*}
and $\pi_X^*L = A + E$ hold. 
From this estimate 
it is also possible to deduce that $A$ is big for a sufficiently large $m$, 
because the above diagram for $\iota$ is also valid for the identity map. 
The proof of the reversed inequality is not hard. 
\end{proof}

With the proof of Theorem \ref{Fujita approximation}, 
we finally get to our goal of this subsection.
Observe that one can approximate $\MA(P_{X|Z}\varphi)$ and $\Vol_{X|Z}(L)$ at the same time 
taking takes suitable modifications. 

{\em Proof of Theorem \ref{Main Theorem}.}
Since $\iota$ is a closed embedding and $\iota(Z) \nsubseteq \mathbb{B}_+(L)$, 
we may assume $\widetilde{\iota}$ is also an embedding and 
$\widetilde{\iota}(\widetilde{Z}) \nsubseteq \mathbb{B}_+(A')$. 
By the semiampleness and the bigness of $A'$, 
there exists a smooth semipositive form $\theta_A$ in $c_1(A)$ such that     
\begin{equation*}
\begin{split}
    \Vol_{\widetilde{X}|\widetilde{Z}}(A) 
  &= \Vol_{\widetilde{Z}}({\widetilde{\iota}}^*A) = \int_{\widetilde{Z}}({\widetilde{\iota}}^* \theta_A)^p \\
  &= \int_{\widetilde{Z}} \big\langle ({\widetilde{\iota}}^*( \theta_A + [E]))^p\big\rangle.  
\end{split}
\end{equation*} 
The last equality is a consequence of the non-pluripolarity of the Monge-Amp\`{e}re product. 
${\widetilde{\iota}}^*(\theta_A + [E])$ and ${\widetilde{\iota}}^*(\gamma + [F])$ 
are in the same class so that 
one can apply Theorem \ref{comparison theorem} to (\ref{condition of modification 1}), 
to deduce the following: 
\begin{equation*}
\begin{split}
    \int_{\widetilde{Z}} \big\langle ({\widetilde{\iota}}^*( \theta_A + [E]))^p\big\rangle         
  & \geq \int_{\widetilde{Z}} \big\langle ({\widetilde{\iota}}^*( \gamma + [F]))^p \big\rangle \\
  & = \int_{\widetilde{Z}} \big\langle (\pi_Z^* \iota^* T_m)^p \big\rangle 
    = \int_Z \big\langle ( \iota^*T_m)^p \big\rangle.  
\end{split}
\end{equation*} 
For an arbitrary $ \varepsilon >0 $, 
the proof of Theorem \ref{convergence of FS} shows 
\begin{equation*}
  \int_Z \big\langle ( \iota^*T_m)^p \big\rangle
  \geq \int_Z \big\langle(\iota^*T)^p \big\rangle - \varepsilon 
\end{equation*}
if we take $m$ sufficiently large. 
This implies  
\begin{equation*}
  \Vol_{X|Z}(L) \geq \Vol_{\widetilde{X}|\widetilde{Z}}(A) \geq \int_Z \MA(\iota^* P_X \varphi)
\end{equation*}
so combining this inequality 
with Theorem \ref{comparison for equilibrium} and Corollary \ref{half of Fujita}, 
we finally get the identity 
\begin{equation}\label{restricted volume can be written by MA}
  \Vol_{X|Z}(L) = \int_Z \MA(P_{X|Z} \varphi). 
\end{equation}
With this identity and Proposition \ref{upper bound 4}, 
Theorem \ref{Main Theorem} is now concluded 
from Lemma 2.2 in \cite{Ber06} which is shown by basic measure theory.
$\hfill \Box$

%４restricted volumeの積分表示%
\section{Integral representations for the restricted volume}

In this section, we discuss several integral representation of the restricted volume. 
\begin{thm}\label{integral representation for restricted volume}
   Let $Z \subseteq X$ be a (possibly singular) subvariety of $X$ and assume 
  $\iota(Z) \nsubseteq \mathbb{B}_+(L)$. Then the following holds. 
 \begin{align*}
  \Vol_{X|Z}(L) 
  &=\int_{Z_{\mathrm{reg}}} \MA(P_{X|Z_{\mathrm{reg}}} \varphi) 
  =\int_{Z_{\mathrm{reg}}} \MA((\iota|_{Z_{\mathrm{reg}}})^*P_X \varphi) \\
  & =\sup_T \int_{Z_{\mathrm{reg}}} \big\langle ((\iota|_{Z_{\mathrm{reg}}})^*T)^p \big\rangle 
  = \int_{Z_{\mathrm{reg}}} \big\langle ((\iota|_{Z_{\mathrm{reg}}})^*T_{\min})^p \big\rangle 
  = \int_{X \setminus \mathbb{B}_+(L)}(T_{\min})^p \wedge \left[Z\right], 
 \end{align*}
  where $T$ runs through all the closed positive currents in $c_1(L)$, 
  with small unbounded loci not contained in $\iota(Z)$. 
  We denote by $T_{\min}$ a minimum singular current in $c_1(L)$ and  
  denote by $Z_{\mathrm{reg}}$ the regular locus of $Z$. 
  The last integrand is defined as a closed positive current 
  on $X \setminus \mathbb{B}_+(L)$ 
  in the manner of Bedford-Taylor, 
  and $\left[Z\right]$ denotes the closed positive current defined by $Z$. 
\end{thm}
\begin{proof} 
  First assume $Z$ is smooth. 
  The first two identities are nothing but (\ref{restricted volume can be written by MA}) and Theorem \ref{comparison for equilibrium}. 
  The second two are consequences of Theorem \ref{comparison theorem}. 
  Let us prove the last identity. 
  Note that the trivial extension of the current 
  $(T_{\min})^p \wedge \left[Z\right]$ to $X$ 
  is a closed positive current and has finite mass by Skoda's extension theorem. 
  Fix a Borel function $\psi$ such that $T_{\min}=\theta + dd^c\psi $. By induction on $p$, we are going to prove that  
  \begin{equation}\label{restriction}
   \int_{X \setminus \mathbb{B}_+(L)} 
    \rho (\theta+dd^c \psi)^p \wedge \left[Z\right] 
   = \int_{Z \setminus \iota^{-1}(\mathbb{B}_+(L))} \iota^*\rho (\iota^*\theta+dd^c \psi)^p  
  \end{equation}
  for any Borel function $\rho$ on $X$. 
 The case $p=0$ is trivial. Assume this is true for $p-1$. 
 First fix a {\em smooth } function $\rho$ on $Z$. 
 Take some $\chi_k \in C^{\infty}_0 (X \setminus \mathbb{B}_+(L))$ $(k=1,2,3, \dots )$
 such that 
 $\chi_k \equiv 1$ outside of the $1/k$-neighborhood of $\mathbb{B}_+(L)$.  
 Then 
 \begin{equation*}
 \begin{split}
 & \int_{X \setminus \mathbb{B}_+(L)} 
    \chi_k \rho (\theta + dd^c\psi) 
     \wedge (\theta + dd^c\psi)^{p-1} \wedge \left[Z\right] \\
 &=\int_{X \setminus \mathbb{B}_+(L)} 
    \chi_k \rho \theta \wedge (\theta + dd^c\psi)^{p-1} \wedge \left[Z\right] 
     + \psi dd^c(\chi_k \rho) 
        \wedge (\theta + dd^c\psi)^{p-1} \wedge \left[Z\right] \\
 &=\int_{Z \setminus \iota^{-1}(\mathbb{B}_+(L))} 
    \iota^* (\chi_k \rho \theta) \wedge (\iota^*\theta + dd^c \iota^* \psi)^{p-1} 
     + \iota^*\psi dd^c \iota^*(\chi_k \rho) 
     \wedge (\iota^*\theta + dd^c \iota^*\psi)^{p-1} \\ 
 \end{split}
 \end{equation*}
 by the induction hypothesis. This equals to 
 \begin{equation*}
 \begin{split}
 &\int_{Z \setminus \iota^{-1}(\mathbb{B}_+(L))} 
    \iota^* (\chi_k \rho \theta) \wedge (\iota^*\theta + dd^c \iota^* \psi)^{p-1} 
     + \iota^*(\chi_k \rho) dd^c\iota^*\psi 
      \wedge (\iota^*\theta + dd^c \iota^*\psi)^{p-1}. 
 \end{split}
 \end{equation*}
 Letting $k \to \infty$, we get (\ref{restriction})
 by the Lebesgue convergence theorem. 
 The general case follows from the density. 
  
  The assumption $Z$ is smooth can be dropped if we consider a resolution of singularities,  
  because the Monge-Amp\`{e}re measure has no mass on any closed proper algebraic subset.  
  For instance, let us prove the first identity. 
  Definition of $P_{X|Z_{\mathrm{reg}}} \varphi $ is the same as (\ref{definition of FS}). 
  If we take a resolution of singularities, 
  $\Vol_{X|Z}(L)=\Vol_{\widetilde{X}|\widetilde{Z}}(\pi_Z^*L)$ holds. 
  It is enough to show 
  $P_{\widetilde{X}|\widetilde{Z}} \pi_Z^*\varphi = \pi_Z^*P_{X|Z_{\mathrm{reg}}} \varphi$ 
  in the regular locus of $\pi_Z$.  
  $P_{\widetilde{X}|\widetilde{Z}} \pi_Z^*\varphi \leq \pi_Z^*P_{X|Z_{\mathrm{reg}}} \varphi$ is trivial 
  and the converse inequality follows 
  by the Riemann-type extension theorem for psh functions. 
  Other identities above are shown in the same manner. 
\end{proof}
\begin{rem}
  The assumption that $T$ has a small unbounded locus can be dropped 
  since we may define the {\em non-pluripolar} Monge-Amp\`{e}re product for any $\theta$-psh function  
  and approximate it by the sequence of minimal singular $\theta$-psh 
  (As in the proof of Proposition 1.20 in \cite{BEGZ08}). 
  $\hfill \Box$
\end{rem}     
 The last representation shows that 
 $\Vol_{X|Z}(L)$ is independent of $L$ in the same first Chern class. 
 This result was already proved in \cite{ELMNP09} algebraically.  
\begin{cor}\label{restricted volume is numerical}
   $\Vol_{X|Z}(L)$ is determined only by $Z \subseteq X$ and $c_1(L)$.
\end{cor}
Further, these representations of the restricted volume do not need sections of $L$ hence 
we can extend the definition of restricted volumes to any class. 
\begin{dfn}\label{restricted volume for class}
  For any big class $ \alpha \in  H^{1,1}(X;\mathbb{R})$ and subvariety $Z \subseteq X$, 
  we define the restricted volume as follows. 
  \begin{align*}
    \Vol_{X|Z}(\alpha) 
  &:=\int_{Z_{\mathrm{reg}}} \MA(P_{X|Z_{{\mathrm{reg}}}} \varphi) 
   =\int_{Z_{\mathrm{reg}}} \MA((\iota|_{Z_{\mathrm{reg}}})^*P_X \varphi) \\
  &=\int_{Z_{\mathrm{reg}}} \big\langle ((\iota|_{Z_{\mathrm{reg}}})^*T_{\min})^p \big\rangle 
   =\sup_T \int_{Z_{\mathrm{reg}}} \big\langle ((\iota|_{Z_{\mathrm{reg}}})^*T)^p \big\rangle 
   =\int_{X \setminus \mathbb{B}_+(\alpha)}(T_{\min})^p \wedge \left[Z\right], 
 \end{align*}
 where T runs through all the closed positive currents in $\alpha$, 
 with small unbounded loci not contained in $\iota(Z)$. 
 $\hfill \Box$
\end{dfn}
For the definitions of the bigness and the augmented base locus for an arbitrary class, see \cite{BEGZ08}. 
Note that the regularity of $P_X \varphi$ for a general class $\alpha$ is already shown in \cite{BD09}. 
We will prove the regularity of $P_{X|Z} \varphi$ for the class $c_1(L)$ 
in section 5. 
But for a general $\alpha$, the corresponding regularity result seems unknown. 
The second identity in the above definition is true   
since it is easily seen that $P_{X|Z} \varphi$ has a small unbounded locus even in this case 
and the proof of Theorem \ref{comparison for equilibrium} is still valid.  
The another identities can be proved totally the same as in the case $\alpha = c_1(L)$.   

In the end of this subsection, 
we give the representation of restricted volumes 
via so-called moving intersection number. 
By definition, the moving intersection number counts 
the number of points where $Z$ and a general divisor $D \in \abs{mL}$ intersects 
outside of the base locus. 
We denote it by $\langle (mL)^p , Z \rangle $. 
It is already known that 
$\Vol_{X|Z}(L) = \lim_{m \to \infty} m^{-p} \big\langle (mL)^p, Z \big\rangle$ 
(see \cite{ELMNP09}, Theorem 2.13).  
The refinement of this result is now obtained.                        
\begin{thm}\label{moving intersection number description} 
In the situation of Theorem \ref{integral representation for restricted volume}, 
 \begin{align*}
   \Vol_{X|Z}(L)&= \lim_{m \to \infty} \int_Z \MA(u_m) \\
            &= \lim_{m \to \infty} \frac{\big\langle (mL)^p, Z \big\rangle}{m^p} =: \norm{L^p. Z}. 
 \end{align*}
\end{thm} 
\begin{proof}
  The first identity is a direct consequence of Theorem \ref{convergence of FS}. 
  The second is easily seen by taking a log resolution of $\abs{mL}$. 
  In fact the second identity holds before taking limit. 
  Notation in the third identity follows \cite{ELMNP09}. 
\end{proof}

%拡張定理%
\section{$L^2$-extension theorem from a subvariety}
 
 In this section, we state the desired $L^2$-extension theorem for our purpose 
 and give a proof. 
 
 Let us first fix notations. Given a holomorphic Hermitian vector bundle $E$ 
 with a metric $h_E$ on a K\"{a}hler manifold $X$, 
 we denote its Chern curvature tensor by $c(E)$. 
 That is, $c(E) := \sqrt{-1}D^2$ where $D$ denotes 
 the exterior covariant derivative associated to the Chern connection of $(E,h_E)$. 
 $c(E)$ is an $E^* \otimes E$-valued real $(1,1)$-form  
 and defines a Hermitian form on $T_{X,x} \otimes E_x$ \ ($x \in X$) as follows: 
 \begin{equation*} 
   H( t_1\otimes e_1, t_2 \otimes e_2) 
   := \big(c(E)(t_1,\sqrt{-1}t_2)e_1|e_2\big) \ \ \ \ \ 
   \text{ for } \ t_1,t_2 \in T_{X,x}, \ e_1,e_2 \in E_x.   
 \end{equation*}
 Here $( \ \ | \ \ )$ is defined by $h_E$. 
 Recall that $c(E)$ is said to be semipositive in the sense of Nakano 
 if $H$ is semipositive everywhere in $X$. 
 And we denote it by $c(E) \geq_{\mathrm{Nak}} 0$. 
 If a K\"{a}hler metric $\omega$ is fixed, 
 $c(E)$ also defines a Hermitian form on $(\pigwedge^{p.q} T_{X,x}^*) \otimes E_x$ 
 as follows: 
 \begin{equation*}
  \theta ( \alpha , \beta ):= 
    \big( [c(E),\Lambda]\alpha|\beta \big) 
    \ \text{ for } \ \alpha , \beta \in (\pigwedge^{p,q}T_{X,x}^*) \otimes E_x 
    \ \ \ \ \ ( x \in X ) ,
 \end{equation*} 
 where $\Lambda$ denotes the formal adjoint operator of the multiplication by $\omega$. 
 It is known that if $p=n$ and $c(E)$ is semipositive in the sense of Nakano, 
 $\theta$ defines a semipositive Hermitian form. 
 We will use the following norm: 
 \begin{equation*}
   \abs{\alpha}^2_{\theta}
   = \inf 
       \Bigg\{ M \geq 0 \ \Bigg| \begin{matrix} 
                  \ \abs{(\alpha | \beta)}^2 
                  \leq M \cdot \theta(\beta , \beta ) \\ 
                  \ \text{ for any } \ \beta \in (\pigwedge^{n,q} T_{X,x}^*) \otimes E_x\\ 
                                 \end{matrix} \Bigg\}
  \ \in [0,+\infty] 
 \end{equation*}
 for $\alpha \in (\pigwedge^{n,q} T_{X,x}^* ) \otimes E_x$. 
\begin{thm}\label{L2 extension 1}
 Let $Z$ be a $p$-dimensional submanifold of a $n$-dimensional K\"{a}hler manifold $X$ 
 with its  K\"{a}hler form $\omega$, $K$ a compact subset of $X$. 
 Then there exist constants $ N=N(Z, K)>0 $ and $ C=C(Z, K)>0 $ such that
 the following holds. 

 Fix any complete K\"{a}hler open set $\Omega \subseteq X$ contained in $K$, 
 a holomorphic vector bundle $E \to X$ with a smooth Hermitian metric $h_E$
 whose Chern curvature satisfying 
  \begin{equation*}
    c(E) \geq_{\mathrm{Nak}} N \cdot \mathrm{id}_E \ \ \ \ \ \text{ on $\Omega$, }
  \end{equation*}
 and $f \in H^0(Z \cap \Omega , \mathcal{O}(K_X \otimes E))$. 
 Then we have a section $F \in H^0(\Omega , \mathcal{O}(K_X \otimes E))$ which satisfies 
 $F|_{Z \cap \Omega }=f$ and
  \begin{equation*}
         \int_{\Omega} \abs{F}_{h_E}^2dV_{\omega, X} 
  \leq C \int_{Z \cap \Omega} \abs{f}_{h_E}^2dV_{\omega, Z}. 
  \end{equation*}
\end{thm}
Although the following proof of this theorem is almost the same as the proof of 
{\em ``the Ohsawa-Takegoshi-Manivel $L^2$-extension theorem''} in \cite{Dem00}, 
we describe it for account of the proof of Theorem \ref{L2 extension 2}. 
The difference from \cite{Dem00} is that 
we deal with arbitrary submanifolds and general vector bundles 
while we give up sharp estimates. 

\begin{proof}
 There exists some $G \in C^{\infty}(\Omega,K_X \otimes E)$ such that 
 \begin{equation*}
   G|_{Z \cap \Omega} = f, \ \ \ \ \ (\bar{\partial}G)|_{Z \cap \Omega} = 0. 
 \end{equation*}
 Fix a smooth cut-off function $\rho: \mathbb{R} \to [0,1]$ satisfying 
 \begin{align*}
   \rho (t) :=  \Bigg\{ \begin{matrix} \ 1 \ \ ( t \leq  \frac{1}{2}) \\
                                       \ 0 \ \ ( t \geq  1 ) \\ 
                    \end{matrix}
        & \ \ \ \ \ \abs{\rho'} \leq 3. 
 \end{align*} 
 Then we set as follows: 
  \begin{align*}
   & G_{\varepsilon} := \rho \bigg(\frac{e^{\psi}}{\varepsilon}\bigg)\cdot G \\
   & g_{\varepsilon} := \bar{\partial}G_{\varepsilon} 
       = \underbrace{\bigg(1 + \frac{e^{\psi}}{\varepsilon}\bigg)
                        \rho'\bigg(\frac{e^{\psi}}{\varepsilon}\bigg)
                        \bar{\partial}\psi_{\varepsilon}\wedge G
                        }_{g^{(1)}_{\varepsilon}}
         + \underbrace{\rho \bigg(\frac{e^{\psi}}{\varepsilon}\bigg)\bar{\partial}G
                          }_{g^{(2)}_{\varepsilon}}, 
  \end{align*}
 where 
  \begin{align*}
   & \psi_{\varepsilon} := \log (\varepsilon + e^{\psi})  \ \ \ \ \ 
     \Bigg( \Leftrightarrow  
            1+\frac{e^{\psi}}{\varepsilon} 
           = \frac{e^{\psi_{\varepsilon}}}{\varepsilon}
     \Bigg) \\ 
   & \psi := \log \sum_{\alpha} \chi^2_{\alpha} \sum_{i=p+1}^{n} \abs{z_{\alpha,i}}^2, 
       \ \ \ \ \ \varepsilon >0. 
  \end{align*}
 Here we choose a locally finite system of local coordinates 
 $\{z_{\alpha,1},...,z_{\alpha,n}\}_{\alpha}$ so that 
  \begin{equation*}
    Z \cap U_{\alpha} = \{ z_{\alpha,p+1}= \cdots =z_{\alpha, n} = 0 \} 
  \end{equation*}
 hold and choose a smooth function $\chi_{\alpha}$ so that the following hold.  
  \begin{equation*}
    \Supp \chi_{\alpha} \subseteq U_{\alpha}, 
    \ \ \sum_{\alpha} \chi_{\alpha}^2 > 0, 
    \ \ \text{ and } 
    \ \ \sum_{\alpha} \chi_{\alpha}^2 \sum_{i=p+1}^{n} \abs{z_{\alpha,i}}^2 < e^{-1}  
    \ \ \ \text{ in $X$. }
 \end{equation*} 
 This $\psi$ satisfies the following condition (see \cite{Dem82}, Proposition 1.4). 
 \begin{itemize}
    \setlength{\itemsep}{0pt}
   \item[$(1)$]
     $\psi \in C^{\infty}(X \setminus Z)\cap L^1_{\mathrm{loc}}(X)$ \\
     $\psi < -1$ in $X$, $\psi \to -\infty$ around $Z$. 
   \item[$(2)$]
     $e^{-(n-p)\psi}$ is {\em not} integrable around any point of $Z$
   \item[$(3)$]
     There exists a smooth real $(1,1)$-form $\gamma$ in $X$ such that \\
     $\sqrt{-1}\partial\bar{\partial}\psi \geq \gamma$ holds in $X \setminus Z$. 
 \end{itemize}
 
 If the equation
  \begin{equation*}
   \begin{cases}
      \bar{\partial}u_{\varepsilon} = \bar{\partial}G_{\varepsilon} 
      \ \text{ in } \ \Omega \\
      \abs{u_{\varepsilon}}^2 e^{-(n-p)\psi} 
      \ \text{ is locally integrable around } \ Z \\ 
   \end{cases}
  \end{equation*}
 has been solved, $u_{\varepsilon} = 0 $ on $Z$ holds by the above condition
 hence the sequence $ \{G_{\varepsilon} - u_{\varepsilon} \}_{\varepsilon}$ 
 is expected to converge to what we want. 
 This is our strategy.   

 To solve $\bar{\partial}$-equations, we quote the following from \cite{Dem00}. 
 \begin{thm}[Ohsawa's modified $L^2$-estimate. \cite{Dem00}, Proposition 3.1]\label{Ohsawa L2 estimate}
  Let $X$ be a complete K\"{a}hler manifold 
  with a K\"{a}hler metric $\omega$ ($\omega$ may not be necessarily complete), 
  $E \to X$ a holomorphic Hermitian vector bundle. 
  Assume that there exist some smooth functions $a,b>0$ and if we set 
   \begin{align*}
    & c'(E) := 
           a \cdot c(E) 
             - \sqrt{-1}\partial\bar{\partial}a 
             - \sqrt{-1}b^{-1}\partial a \wedge \bar{\partial}a  \\
    & \theta' ( \alpha , \beta ) := 
               \big( [c'(E),\Lambda]\alpha|\beta \big) 
                 \ \text{ for } 
               \ \alpha, \beta \in (\pigwedge^{n,q}T_{X,x}^*) \otimes E_x 
               \ \ \ \ \ ( x \in X), 
   \end{align*}
  it holds that 
   \begin{equation*}
    \theta' \geq 0 \ \text{ on } \ (\pigwedge^{n,q}T_{X,x}) \otimes E_x 
    \ \ \ \ \ \text{ for any $ x \in X$}. 
   \end{equation*}
Then we have the following. 

 For any $g \in L^2(X, (\pigwedge^{n,q}T_X^*) \otimes E )$ with $\bar{\partial}g=0$ and 
  \begin{equation*}
    \int_X \abs{g}^2_{{\theta}'}dV_{\omega, X} < +\infty, 
  \end{equation*}
 there exists a section $u \in L^2(X, (\pigwedge^{n,q-1}T_X^*) \otimes E)$ 
 with $\bar{\partial}u = g$ such that 
  \begin{equation*}
    \int_X (a+b)^{-1}\abs{u}^2dV_{\omega,X} \leq 2 \int_X \abs{g}^2_{{\theta}'}dV_{\omega,X}.  
  \end{equation*}

\end{thm}

 Let us go back to the proof of Theorem \ref{L2 extension 1}. 
 First, we are going to compute 
 \begin{equation}\label{modified Chern curvature}
  \theta_{\varepsilon}' := 
           \big[ a_{\varepsilon}(c(E) + (n-p) \sqrt{-1} \partial \bar{\partial} \psi )
                        - \sqrt{-1} \partial\bar{\partial} a_{\varepsilon} 
                        - b_{\varepsilon}^{-1} \sqrt{-1}\partial a_{\varepsilon} \wedge \bar{\partial} a_{\varepsilon}, \Lambda 
           \big]. 
 \end{equation} 
 ($a_{\varepsilon}, b_{\varepsilon}$ will be defined in the following. )
 If we set 
 \begin{equation*}
  a_{\varepsilon} := \chi_{\varepsilon}(\psi_{\varepsilon}) > 0
 \end{equation*}
 for some smooth function $\chi_{\varepsilon}$, it can be computed as: 
 \begin{align*}
  \partial a_{\varepsilon}                        
   & = \chi_{\varepsilon}'(\psi_{\varepsilon}) \partial \psi_{\varepsilon}, \\
  \sqrt{-1} \partial \bar{\partial} a_{\varepsilon} 
   & = \chi_{\varepsilon}'(\psi_{\varepsilon}) \sqrt{-1} \partial \bar{\partial} \psi_{\varepsilon}
      + \chi_{\varepsilon}''(\psi_{\varepsilon}) \sqrt{-1} \partial \psi_{\varepsilon} \wedge \bar{\partial} \psi_{\varepsilon} \\ 
   & = \chi_{\varepsilon}'(\psi_{\varepsilon} ) \sqrt{-1} \partial \bar{\partial} \psi_{\varepsilon}
      + \frac{\chi_{\varepsilon}''(\psi_{\varepsilon})}{\chi_{\varepsilon}'(\psi_{\varepsilon})^2} \sqrt{-1} \partial a_{\varepsilon} \wedge \bar{\partial} a_{\varepsilon} 
 \end{align*} 
 so comparing with (\ref{modified Chern curvature}), it is natural to set 
 \begin{equation*}
  b_{\varepsilon} := 
          - \frac{\chi_{\varepsilon}'(\psi_{\varepsilon})^2}{\chi_{\varepsilon}''(\psi_{\varepsilon})} \ \ (>0). 
 \end{equation*} 
 And we finally define 
 \begin{equation*}
  \chi_{\varepsilon}(t) := \varepsilon -t + \log(1-t). 
 \end{equation*}
 Then for sufficiently small $\varepsilon>0$, we have 
 \begin{align*}
  & a_{\varepsilon} 
    \ \geq \ \varepsilon - \log (\varepsilon + e^{-1}) \ \geq \ 1 \\
  & \sqrt{-1} \partial \bar{\partial} a_{\varepsilon} + b_{\varepsilon}^{-1} \sqrt{-1} \partial a_{\varepsilon} \wedge \bar{\partial} a_{\varepsilon} 
   =\chi_{\varepsilon}'(\psi_{\varepsilon}) \sqrt{-1} \partial \bar{\partial} \psi_{\varepsilon}
    \ \leq \ - \sqrt{-1} \partial \bar{\partial} \psi_{\varepsilon} 
 \end{align*}
 hence 
 \begin{equation*}
  \theta_{\varepsilon}' \ \geq \ \big[ c(E) + (n-p)\sqrt{-1} \partial \bar{\partial} \psi + \sqrt{-1} \partial \bar{\partial} \psi_{\varepsilon}, \Lambda \big]. 
 \end{equation*}
 On the other hand, simple computations show:  
 \begin{align*} 
  \partial \psi_{\varepsilon} 
  &= \frac{e^{\psi}}{\varepsilon + e^{\psi}} \partial \psi, \\
  \sqrt{-1}\partial \bar{\partial} \psi_{\varepsilon} 
  &= \frac{e^{\psi}}{\varepsilon + e^{\psi}} \sqrt{-1} \partial \bar{\partial} \psi
   + \frac{e^{\psi}}{\varepsilon + e^{\psi}} \sqrt{-1} \partial \psi \wedge \bar{\partial} \psi 
   - \frac{e^{2\psi}}{(\varepsilon + e^{\psi})^2} \sqrt{-1} \partial \psi \wedge \bar{\partial} \psi \\ 
  &= \frac{e^{\psi}}{\varepsilon + e^{\psi}} \sqrt{-1} \partial \bar{\partial} \psi 
    + \frac{\varepsilon}{e^{\psi}} \sqrt{-1} \partial \psi_{\varepsilon} \wedge \bar{\partial} \psi_{\varepsilon}. 
 \end{align*}
 Therefore, by the compactness of $K$, there exists a constant $N(Z,K)>0$ such that 
 \begin{equation}\label{Nakano positivity}
  c(E) \geq_{\mathrm{Nak}} N \cdot \mathrm{id_E} \ \ \ \ \ \text{ on } \Omega 
 \end{equation}
 implies 
 \begin{equation}\label{theta is positive}
  \theta_{\varepsilon}' \geq 0 \ \text{ on } \  
    (\pigwedge^{n,1} T_{X,x}^*) \otimes E_x \ \ \ \ \ \text{ for all } x \in \Omega 
 \end{equation}
 and eigenvalues of ${\theta}_{\varepsilon}'$ 
 are bounded from below by a positive constant 
 (uniformly with respect to $\varepsilon$) near $Z \cup \Omega$.

 Next we will estimate $\bar{\partial}\psi_{\varepsilon}$ 
 by $\abs{\cdot}_{{\theta}_{\varepsilon}'}$. 
 Fix arbitrary $\alpha,\beta \in (\pigwedge^{n,1} T_{X,x}^*) \otimes E_x$. 
 By definition, 
 \begin{equation*} 
  \abs{\bar{\partial}\psi_{\varepsilon} \wedge \alpha}^2_{{\theta}_{\varepsilon}'}
  = \inf 
    \Bigg\{ M \geq 0 \ \Bigg| \begin{matrix} 
                               \ \abs{(\bar{\partial}\psi_{\varepsilon}\wedge \alpha | \beta)}^2 
                                 \leq M \cdot \big( [c_{\varepsilon}'(E)\Lambda]\beta|\beta \big) \\ 
                               \ \text{ for any } \ \beta \in (\pigwedge^{n,1} T_{X,x}^*) \otimes E_x\\ 
                                       \end{matrix} \Bigg\} 
 \end{equation*} 
 so it is enough to estimate 
 $\abs{(\bar{\partial}\psi_{\varepsilon}\wedge \alpha | \beta)}^2$. 
 This can be done as follows: 
 \begin{equation*} 
 \begin{split}
  &\abs{(\bar{\partial}\psi_{\varepsilon}\wedge \alpha | \beta)}^2
   = \abs{(\alpha | (\bar{\partial}\psi_{\epsilon})^{\sharp}\beta)}^2 \\
  & \leq \abs{\alpha}^2 \cdot \abs{(\bar{\partial}\psi_{\epsilon})^{\sharp}\beta}^2 
   = \abs{\alpha}^2 \big( ( \bar{\partial}\psi_{\varepsilon})(\bar{\partial}\psi_{\varepsilon})^{\sharp}\beta|\beta \big) 
   = \abs{\alpha}^2 \big( [\sqrt{-1} \partial \psi_{\epsilon} \wedge \bar{\partial} \psi_{\epsilon} , \Lambda ] \beta | \beta \big) 
 \end{split}
 \end{equation*}
 by Shwartz' inequality ($\sharp$ denotes taking the formal adjoint of the multiplication operator), 
 and the last term is bounded by 
 \begin{equation*}
 \begin{split} 
  & \frac{e^{\psi}}{\varepsilon} \abs{\alpha}^2
                    \big( [\sqrt{-1}\partial \bar{\partial} \psi_{\varepsilon}
                       - \frac{e^{\psi}}{\varepsilon+e^{\psi}}\sqrt{-1}\partial \bar{\partial} \psi
                      , \Lambda ] \beta| \beta \big) \\
  & \leq  \frac{e^{\psi}}{\varepsilon} \abs{\alpha}^2
                    \big( [c(E) + (n-p)\sqrt{-1} \partial \bar{\partial} \psi 
                           + \sqrt{-1} \partial \bar{\partial} \psi_{\varepsilon}
                           , \Lambda] \beta| \beta \big) \\
  & \leq  \frac{e^{\psi}}{\varepsilon} \abs{\alpha}^2
                    \big( [c'_{\varepsilon}(E), \Lambda] \beta| \beta \big). 
 \end{split}
 \end{equation*} 
 The last inequality is a consequence of (\ref{Nakano positivity}). 
 Thus we may get a desired estimate  
 \begin{equation}\label{key estimate}
  \abs{\bar{\partial} \psi_{\varepsilon} \wedge \alpha}^2_{{\theta}_{\varepsilon}'}
  \leq \frac{e^{\psi}}{\varepsilon} \abs{\alpha}^2. 
 \end{equation}
 
 This time we estimate 
 $g_{\varepsilon} = g_{\varepsilon}^{(1)} + g_{\varepsilon}^{(2)}$ . \\ 
 By (\ref{key estimate}) and 
 $\Supp g_{\varepsilon}^{(1)} \subseteq \{ e^{\psi} < \varepsilon \}$,  
 $g_{\varepsilon}^{(1)}$ can be estimated. Namely, 
 \begin{equation*}
   \int_{\Omega \setminus Z} \abs{g_{\varepsilon}^{(1)}}^2_{{\theta}_{\varepsilon}'}e^{-(n-p)\psi} dV_{\omega, X} 
    \leq 4 \int_{\Omega \setminus Z} \abs{G}^2 \rho'\bigg( \frac{e^{\psi}}{\varepsilon}\bigg)^2e^{-(n-p)\psi} dV_{\omega, X} 
 \end{equation*} 
 holds. Since $e^{\psi} \sim \sum_{i=p+1}^{n} \abs{z_{\alpha, i}}^2 $ 
 on $U_{\alpha}$, thanks to the compactness of $K$ we get: 
 \begin{equation*}
      \limsup_{\varepsilon \to 0} \int_{\Omega \setminus Z} \abs{g_{\varepsilon}^{(1)}}^2_{{\theta}_{\varepsilon}'}e^{-(n-p)\psi} dV_{\omega, X} 
      \leq C \int_{Z \cap \Omega} \abs{f}^2 dV_{\omega, Z} < +\infty. 
 \end{equation*}
 We can also estimate $g_{\varepsilon}^{(2)}$.   
 Note that eigenvalues of ${\theta}_{\varepsilon}'$ are bounded below. 
 Then we get  
 \begin{equation*}
   \int_{\Omega \setminus Z} \abs{g_{\varepsilon}^{(2)}}^2_{{\theta}_{\varepsilon}'}e^{-(n-p)\psi}dV_{\omega, X} \leq O(\varepsilon) < +\infty
 \end{equation*} 
 because we can see that 
 $\displaystyle \abs{g_{\varepsilon}^{(2)}}^2_{{\theta}_{\varepsilon}'} = O(e^{\psi})$ 
 holds in $\Supp g_{\varepsilon}^{(2)} \subseteq \{ e^{\psi} < \varepsilon \} $, 
 by $\bar{\partial} G |_{Z \cap \Omega} = 0$ (using the Taylor expansion).  

 Now we can apply the modified $L^2$-estimate 
 for each $\varepsilon$ in $\Omega \setminus Z$. 
 Note that $\Omega \setminus Z$ is a complete K\"{a}hler manifold (see \cite{Dem82}, 
 Theorem 1.5). 
 There exists a sequence $ \{ u_{\varepsilon} \} \subseteq L^2(\Omega, K_X \otimes E )$ such that 
 \begin{equation*}
  \int_{\Omega \setminus Z} (a_{\varepsilon} + b_{\varepsilon} )^{-1} \abs{u_{\varepsilon}}^2 e^{-(n-p)\psi} dV_{\omega, X} 
  \leq 2 \int_{\Omega \setminus Z} \abs{g_{\varepsilon}}^2_{{\theta}_{\varepsilon}'}e^{-(n-p)\psi} dV_{\omega, X} \ < +\infty
 \end{equation*} 
 holds.  

 Let us estimate the left hand side of the inequality. It can be easily seen that 
 \begin{align*}
   & \psi_{\varepsilon} 
     \ \leq \ \log (\varepsilon + e^{-1}) \ \leq -1 \ + O(\varepsilon) \\
   & a_{\varepsilon} 
     \ \leq \ (1+ O(\varepsilon)) \psi^2_{\varepsilon} \\ 
   & b_{\varepsilon} 
     \ = \ (2-\psi_{\varepsilon})^2  
     \ \leq \ (9+ O(\varepsilon)) \psi^2_{\varepsilon} \\
   & a_{\varepsilon} + b_{\varepsilon} 
     \ \leq \ (10 + O(\varepsilon))\psi_{\varepsilon}^2
     \ \leq \ (10 + O(\varepsilon))(-\log(\varepsilon + e^{\psi}))^2 
 \end{align*}
 and 
 \begin{equation*}
  \int_{\Omega} \frac{\abs{G_{\varepsilon}}^2}{(\varepsilon+e^{\psi})^{(n-p)}(-\log(\varepsilon+e^{\psi}))^2}dV_{\omega, X}
  \leq \frac{M}{(\log \varepsilon)^2} 
 \end{equation*} 
 hold.  
 Therefore, if we set $F_{\varepsilon} := G_{\varepsilon} - u_{\varepsilon}$, it follows: 
 \begin{equation*}
 \begin{split}
  & \limsup_{\varepsilon \to 0} \int_{\Omega \setminus Z} \frac{\abs{F_{\varepsilon}}^2}{(\varepsilon + e^{\psi})^{(n-p)}(-\log(\varepsilon + e^{\psi}))^2}dV_{\omega, X} \\ 
  & \leq \limsup_{\varepsilon \to 0} 
     \Bigg(
     \ 22 \int_{\Omega \setminus Z} \abs{g_{\varepsilon}}^2_{{\theta}_{\varepsilon}'}e^{-(n-p)\psi}dV_{\omega, X} \ + \ \frac{2M}{(\log \varepsilon)^2} 
     \Bigg) \leq C \int_{Z \cap \Omega} \abs{f}^2 dV_{\omega, Z} \ < + \infty. 
 \end{split}
 \end{equation*}
 By construction, $\bar{\partial}F_{\varepsilon} = 0 $ holds on $\Omega \setminus Z$ 
 and in fact also in $\Omega$, thanks to the Riemann extension theorem. 
 
 Finally, Let $\varepsilon \searrow 0$. 
 Then after taking a weakly convergent subsequence, 
 we get a $F \in L^2(\Omega, K_X \otimes E )$ 
 such that $\bar{\partial}F = 0 $ in $\Omega$ and 
 \begin{equation*}
  \int_\Omega \frac{\abs{F}^2}{e^{(n-p)\psi}(-\psi)^2}dV_{\omega, X} \leq C \int_{Z \cap \Omega} \abs{f}^2 dV_{\omega, Z}. 
 \end{equation*} 
 By the compactness of $K$, we get the conclusion. 
\end{proof}

{\em Proof of Theorem \ref{L2 extension 2}.} Since $X$ is projective, 
 we may take a global meromorphic section $\sigma $ of $L$ 
 and may assume   
 $ \Supp ( \div( \sigma )  ) \cap Z \subsetneq Z$.  
 Fix a hypersurface $H \subseteq X$ such that $X \setminus H$ is Stein, 
 $\Supp(\div (\sigma)) \subseteq H$, and 
 $H \cap Z \subsetneq Z$ hold. 
 Then $L|_{X \setminus H} $ is trivial so that 
 we may identify $\varphi$ as a psh function on $X \setminus H$. 

 Let $\psi$ be a smooth exhaustive strictly-psh function in $X \setminus H$ 
 and set $\Omega_k := \{ \psi < k \}$. 
 Since $X \setminus H$ is Stein, 
 there exists a sequence $\varphi_k \in \PSH(\Omega_k)$ 
 satisfying $\varphi_k \searrow \varphi$ 
 (pointwise convergence in $\Omega_k$). 
 Note that $\varphi_k$ does not loss positivity.  

 We apply Theorem \ref{L2 extension 1} to $\Omega' := \Omega_k$ 
 and $E':= K_X^{-1} \otimes E \otimes L$ for each $k$.  
 Then by assumption there are sections 
 $\widetilde{s}_k \in H^0(\Omega_k, \mathcal{O}(K_X \otimes E'))$ 
 such that $\widetilde{s}_k|_{Z \cap \Omega_k} = s $ 
 and 
 \begin{equation}\label{L2 estimate}
 \begin{split}
   \int_{\Omega_k} \abs{\widetilde{s}_k}^2 e^{-\varphi_k} dV_{\omega, X} 
    & \leq C \int_{Z \cap \Omega_k} \abs{s}^2 e^{-\varphi_k} dV_{\omega, Z} \\
    & \leq C \int_Z \abs{s}^2 e^{- \varphi} dV_{\omega, Z}
 \end{split}
 \end{equation} 
 for a constant $C$. If we fix $l \in \mathbb{N}$, 
 there exists some constant $c(l) \leq e^{-\varphi_l}$ in $\Omega_l$ 
 hence we have 
 \begin{equation*}
   c(l) \cdot \int_{\Omega_l} \abs{\widetilde{s}_k}^2 dV_{\omega, X} 
   \leq C \int_Z \abs{s}^2 e^{-\varphi} dV_{\omega, Z}.    
 \end{equation*} 
 Using the diagonal process, we may find a subsequence:
 $\widetilde{s}_{k(i)} \to \widetilde{s}$ 
 (weakly $L^2$-convergent on $X$). 
 By Lemma \ref{weak L2 convergence implies pointwise convergence}, 
 $\bar{\partial} \widetilde{s}_{k(i)} = 0$ 
 implies that this is actually the pointwise convergence so that 
 $\widetilde{s}|_{Z \cap (X \setminus H)} = s $ holds. 
 We can deduce  
 \begin{equation*}
   \int_{X \setminus H} \abs{\widetilde{s}}^2 e^{-\varphi} dV_{\omega, X} 
    \leq C \int_Z \abs{s}^2 e^{-\varphi} dV_{\omega, Z} 
 \end{equation*}
 by (\ref{L2 estimate}) and by the lower-semicontinuity of $L^2$-norm.  
 $\widetilde{s} $ can be extended to $X$ by the Riemann extension theorem 
 and thus we conclude the theorem.  
$\hfill \Box$
\begin{lem}\label{weak L2 convergence implies pointwise convergence}
  Let $f_k$, $f$ be holomorphic functions defined in a domain 
 $\Omega \subseteq \mathbb{C}^n$. 
  Assume that the sequence $\{f_k\}$  weakly $L^2$-converges to $f$. 
  Then $\{f_k\}$ converges to $f$ pointwise in $\Omega$. 
\end{lem}
\begin{proof}
  Fix any point $x \in \Omega$. 
  Taking $\chi \in C^{\infty}_0(\Omega)$ with $\chi \equiv 1$ near $x$, 
  we have: 
  \begin{equation*} 
    f_k(x) = \int_{\zeta \in \Omega} K^{n,0}_{\mathrm{BM}} (x, \zeta) \wedge \bar{\partial} \chi (\zeta) \wedge f_k(\zeta)
           \to \int_{\zeta \in \Omega} K^{n,0}_{\mathrm{BM}} (x, \zeta) \wedge \bar{\partial} \chi (\zeta) \wedge f(\zeta) = f(x) 
  \end{equation*} 
  by the Koppelman formula. 
  Here $K^{p,q}_{\mathrm{BM}}$ denotes the $(p,q)$-part of the Bochner-Martinelli kernel.    
\end{proof}
 
 {\bf Acknowledgments.}
The author would like to express his gratitude 
to his advisor Professor Shigeharu Takayama for his warm encouragements, 
suggestions and reading the drafts. 
The author also would like to thank Professor Takeo Ohsawa 
for several helpful comments concerning 
the $L^2$-extension theorem (Theorem \ref{L2 extension 2}). 
He is indebted to Doctor Shin-ichi Matsumura for valuable discussions in the seminars. 
This research is supported by JSPS Research Fellowships for Young Scientists (22-6742). 
 
%リファレンス%

\end{document}